\documentclass[12pt,reqno]{amsart}
\usepackage{amsmath,amssymb,amsfonts,amsthm,amstext}
\usepackage{url,xspace}
\usepackage{color,graphicx}
\usepackage{graphics}
\usepackage{enumerate}
\usepackage{cleveref}
\usepackage{tikz}
\usepackage{array,multirow}
\usepackage[margin=3.7cm]{geometry}

\usepackage[numbers,sort&compress]{natbib}

\newcommand{\Z}{{\mathbb Z}}
\newcommand{\R}{{\mathbb R}}
\newcommand{\N}{{\mathbb N}}
\renewcommand{\P}{\mathbb{P}}
\newcommand{\E}{{\mathbb E}}
\newcommand{\ind}{{\mathbf 1}}

\newcommand{\eps}{\epsilon}

\newcommand{\eqd}{\stackrel{\mathrm{d}}{=}}
\newcommand{\tod}{\stackrel{\mathrm{d}}{\rightarrow}}

\newcommand{\red}{{\mathcal R}}
\newcommand{\blue}{{\mathcal B}}
\newcommand{\mat}{{\mathcal M}}
\newcommand{\leb}{{\mathcal L}}
\newcommand{\X}{{\mathcal X}}
\newcommand{\A}{{\mathcal S}}

\newcommand{\df}{\textbf}

\newcommand{\bigmid}{\,\big\vert\,}
\newcommand{\Bigmid}{\,\Big\vert\,}

\DeclareMathOperator*{\essinf}{ess\,inf}
\DeclareMathOperator*{\supp}{supp}

\newtheorem{thm}{Theorem}
\newtheorem{lemma}[thm]{Lemma}
\newtheorem{prop}[thm]{Proposition}
\newtheorem{cor}[thm]{Corollary}

\theoremstyle{definition}
\newtheorem{remark}[thm]{Remark}

\newcommand{\rs}{resp.\ }

\newcommand{\e}[1]{\langle #1 \rangle}

\crefname{thm}{Theorem}{Theorems}
\crefname{lemma}{Lemma}{Lemmas}
\crefname{prop}{Proposition}{Propositions}
\crefname{cor}{Corollary}{Corollaries}
\crefname{section}{Section}{Sections}
\crefname{figure}{Figure}{Figures}
\crefname{table}{Table}{Tables}
\crefname{definition}{Definition}{Definitions}
\crefname{construction}{Construction}{Constructions}

\newcommand\bigpar[1]{\bigl(#1\bigr)}

\title{Minimal matchings of point processes}

\date{12 December 2020}

\author{Alexander E.~Holroyd}
\address{Alexander E.~Holroyd, School of Mathematics, University of Bristol, Bristol BS8~1UG, United Kingdom}
\email{a.e.holroyd@bristol.ac.uk}
\author{Svante Janson}
\address{Svante Janson, Department of Mathematics, Uppsala University, PO Box 480,
SE-751~06 Uppsala, Sweden}
\email{svante.janson@math.uu.se}
\author{Johan W\"astlund}
\address{Johan W\"astlund, Department of Mathematical Sciences, Chalmers University of Technology, SE-412~96 Gothenburg, Sweden}
\email{wastlund@chalmers.se}
\thanks{Funded in part by
the Knut and Alice Wallenberg Foundation (SJ and AEH) and the Royal Society (AEH)}


\begin{document}

\begin{abstract}
  Suppose that red and blue points form independent homogeneous Poisson processes of equal intensity in $\R^d$.  For a positive (respectively, negative) parameter $\gamma$ we consider red-blue matchings that locally minimize (respectively, maximize) the sum of $\gamma$th powers of the edge lengths, subject to locally minimizing the number of unmatched points. The parameter can be viewed as a measure of fairness. The limit $\gamma\to-\infty$ is equivalent to Gale-Shapley stable matching.  We also consider limits as $\gamma$ approaches $0$, $1-$, $1+$ and $\infty$.  We focus on dimension $d=1$.  We prove that almost surely no such matching has unmatched points. (This question is open for higher $d$).  For each $\gamma<1$ we establish that there is almost surely a unique such matching, and that it can be expressed as a finitary factor of the points.  Moreover, its typical edge length has finite $r$th moment if and only if $r<1/2$.  In contrast, for $\gamma=1$ there are uncountably many matchings, while for $\gamma>1$ there are countably many, but it is impossible to choose one in a translation-invariant way.  We obtain existence results in higher dimensions (covering many but not all cases).  We address analogous questions for one-colour matchings also.
\end{abstract}

\maketitle

\section{Introduction}

Let $R$ and $B$ be discrete subsets of a metric space.  We call their elements \df{red points} and \df{blue points} respectively.  We are primarily interested in random infinite sets, and in particular the case where $R$ and $B$ are independent homogeneous Poisson processes of equal intensity on $\R^d$.  We will focus especially on dimension $d=1$.

A \df{$2$-colour perfect matching} of $R$ and $B$ is a set $M$ of ordered pairs $\e{r,b}$, called \df{edges}, with $r\in R$ and $b\in B$, such that each red or blue point belongs to exactly one edge.  We wish to address matchings that minimize the distances between matched pairs.  To make this precise, one must decide on the relative weighting of short versus long edges, and also make sense of minimization of infinitely many variables together.

%To address these issues,
We therefore introduce the following concept.  Suppose $R,B\subseteq\R^d$ and let $\gamma\in(0,\infty)$ be a parameter.  We say that a perfect matching $M$ is \df{$\gamma$-minimal} if for every \emph{finite} set of edges $\{\e{r_1,b_1},\ldots,\e{r_n,b_n}\}\subseteq M$ we have
\begin{equation}\label{mm}
\sum_i |r_i-b_i|^\gamma =\min_{\sigma} \sum_i |r_i-b_{\sigma(i)}|^\gamma,
\end{equation}
where $|\cdot|$ denotes the Euclidean norm, and the minimum is over all permutations $\sigma$ of $1,\ldots,n$.  In other words, $M$ \emph{locally} minimizes the total \emph{cost} given by the sum of $\gamma$th powers of the edge lengths.

We also define \df{$\gamma$-minimal} matchings for negative $\gamma\in(-\infty,0)$ in the same way, except that, since $x\mapsto x^\gamma$ is decreasing, we replace $|\cdot|^\gamma$ with $-|\cdot|^\gamma$ on both sides of \eqref{mm} (which is equivalent to replacing ``$\min$" with ``$\max$'').  We also consider matchings given by replacing $|\cdot|^\gamma$ with $\log|\cdot|$ in \eqref{mm}.  We call such matchings \df{$0$-minimal}, because they can be interpreted via the limit $\gamma\to 0$ (applied to finite sets of edges).   One could choose to minimize other functions of distance, but it turns out that powers and logarithms are the only scale-invariant choices.  (We spell out and prove this and other such claims in \cref{defs}.)

We can imagine each point as an agent that wants a partner as close as possible.  The parameter $\gamma$ can then be interpreted as a measure of fairness or altruism.  For large $\gamma$, long edges are heavily penalized, so that costs tend to be shared evenly among many points.  In contrast, for small $\gamma$, points that can match close by tend to do so selfishly, regardless of impact on others.

We will also consider \df{$\infty$-minimal} and \df{$(-\infty)$-minimal} matchings, which arise as the limits $\gamma\to\pm\infty$.  Here, any finite set of edges is required to minimize the length of the longest or the shortest edge, respectively.  Subject to a regularity condition on the point sets (satisfied almost surely by Poisson processes), a matching is ($-\infty$)-minimal if and only if it is \df{stable}, which is defined to mean that there do not exist a red-blue pair that are both strictly closer to each other than their partners.  Stable matching was introduced in the celebrated work of Gale and Shapley \cite{gs} and has been studied in the context of point processes in \cite{hpps} and a number of subsequent articles.  %In particular, stable matchings are almost surely unique in the Poisson setting.
At the other extreme, we sometimes call $\infty$-minimal matchings \df{altruistic}.  Under stable matching, a point exclusively pursues its self-interest, whereas in the altruistic case it promotes the needs of any less fortunate point.
%(Again, we give full details in \cref{defs}.)

%At the ``selfish'' extreme $\gamma\to-\infty$ we obtain \df{stable} matching, which we will also call \df{$(-\infty)$-minimal} matching.  A matching $M$ is stable if and only if there exists no pair $r\in R$ and $b\in B$ for which $|r-b|<\min\{|r-M(r)|,|b-M(b)|\}$ (here $M(x)$ is the partner of the point $x$).  Stable matching was introduced in the celebrated work of Gale and Shapley \cite{gale-shapley} and has been studied in the context of point processes in \cite{hpps} and a number of subsequent articles.  At the other extreme $\gamma\to+\infty$ we obtain another interesting object which we call \df{altruistic} or \df{$\infty$-minimal} matching.  Here, any finite set of edges is required to minimize the longest edge, and, subject to that, minimize the second longest, and so on.
%
The case $\gamma=1$ in dimension $d=1$ is special in that there are many ties: for any $r,r'\in R$ and $b,b'\in B$ with $r<r'<b'<b$, the costs $|r-b|+|r'-b'|$ and $|r-b'|+|r'-b|$ of the two possible matchings are equal.  It is therefore natural to consider the limits $\gamma\uparrow 1$ and $\gamma\downarrow 1$, which amounts to insisting that such ties are always resolved in favour of the first or the second matching respectively.  We call the resulting matchings \df{$(1-)$-minimal} and \df{$(1+)$-minimal} respectively.  All $(1-)$- and $(1+)$-minimal matchings are also $1$-minimal.

See \cref{1d,2d} for pictures of minimal matchings of random points on bounded regions.  A primary motivation is to understand when such pictures can be meaningfully extended to infinite space.

\begin{figure}\centering
  \includegraphics[height=.49\textwidth]{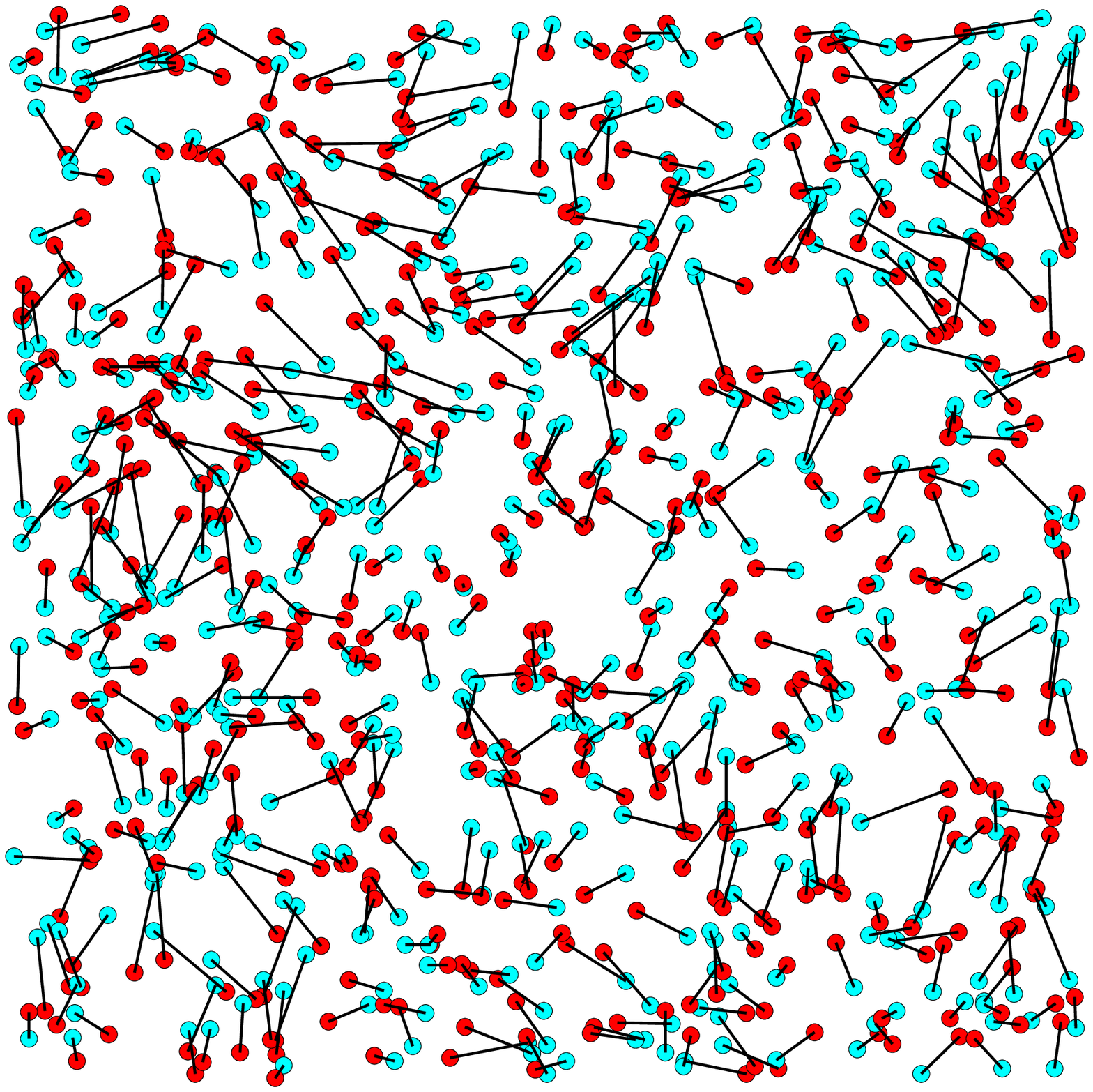}\hfill%
  \includegraphics[height=.49\textwidth]{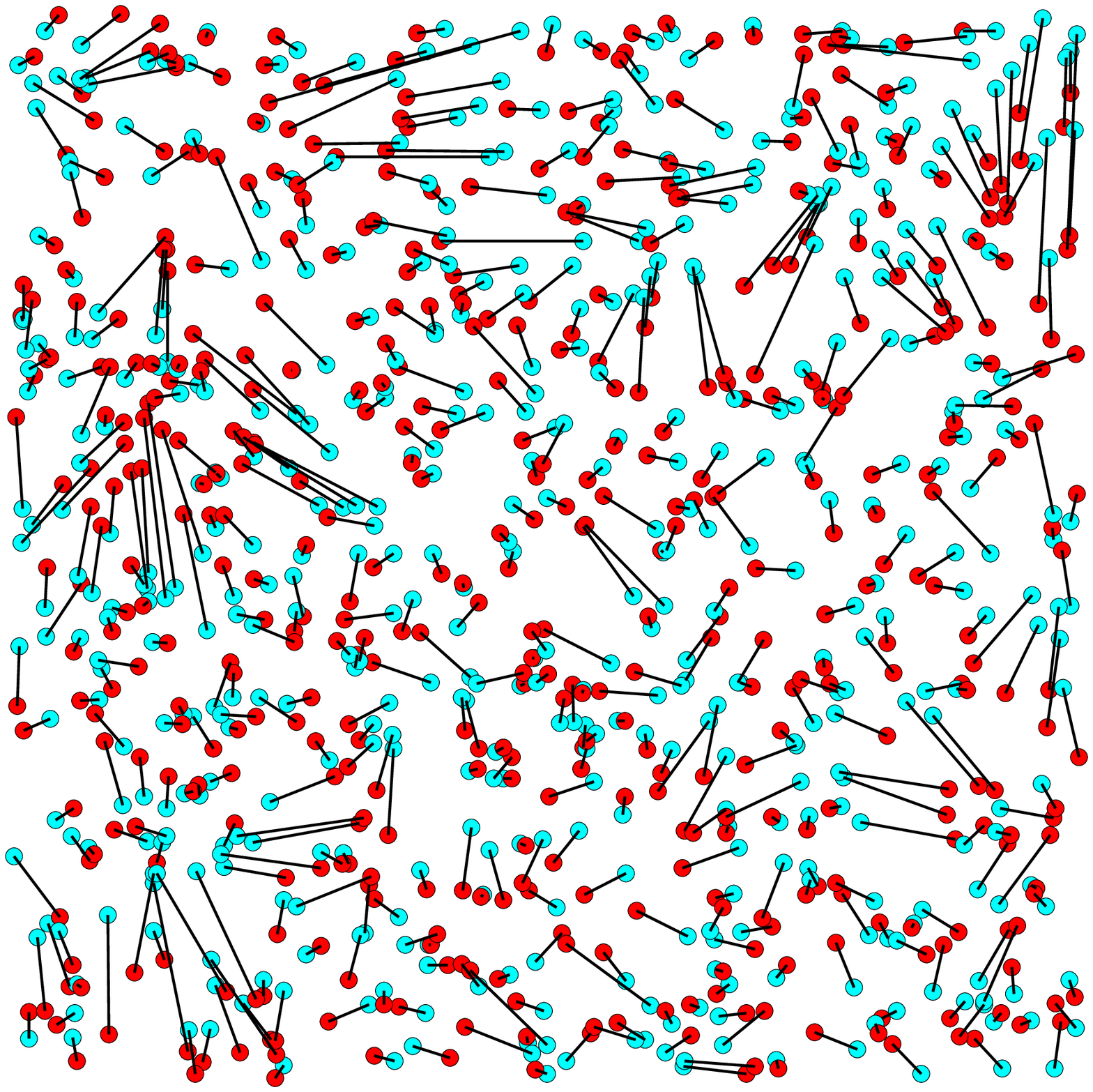}\\[1ex]
  \includegraphics[height=.49\textwidth]{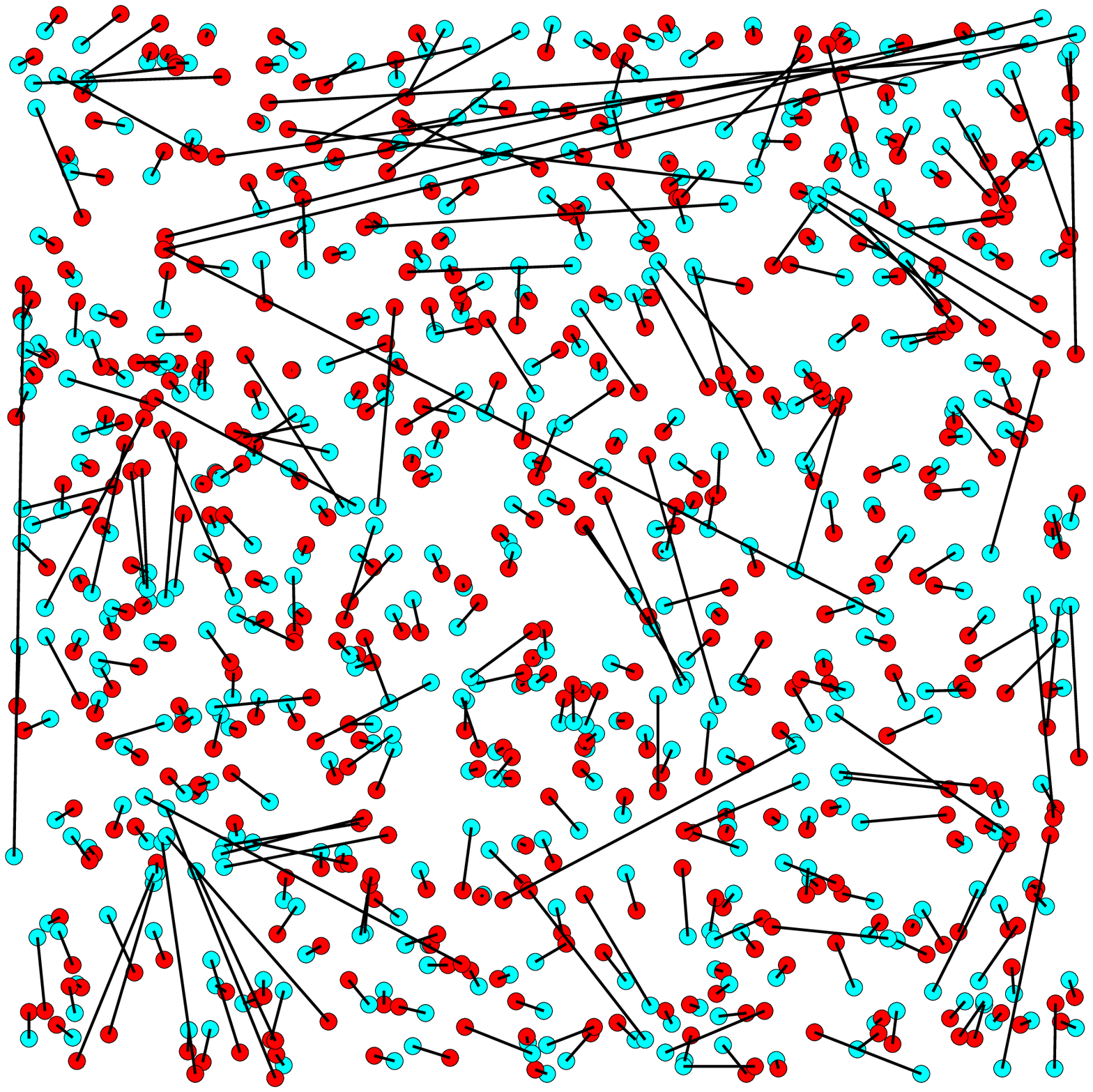}

  \caption{Uniformly random red and blue points in equal numbers on a square, together with $\gamma$-minimal matchings for $\gamma=\infty$ (altruistic; top-left), $1$ (top-right), and $-\infty$ (stable; bottom).  Of these, only the last is known to exist on the infinite plane.}\label{2d}
\end{figure}

\begin{figure}\centering
  \includegraphics[width=.9\textwidth]{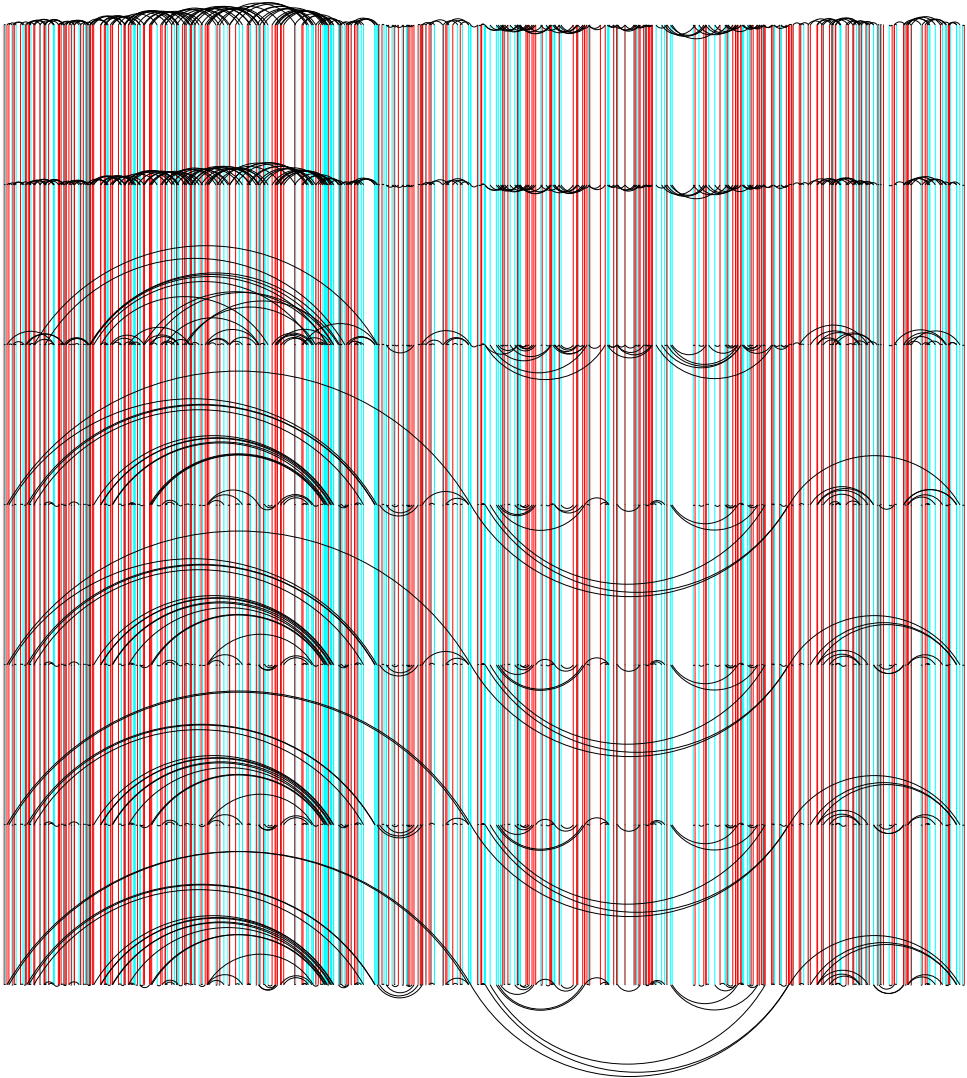}
  \caption{Uniformly random red and blue points in equal numbers on an interval, together with $\gamma$-minimal matchings for $\gamma=3,2,1,0,-1,-2,-3$ (top to bottom).  The points are identical for each matching, and are shown as vertical lines.  Edges are shown as upward or downward arcs depending on whether the red or blue point is on the left.}\label{1d}
\end{figure}

We will address the questions: when do $\gamma$-minimal matchings exist?  When are they unique?  When are they translation-invariant?  When can they be constructed locally, and when is extra randomness needed to do so?  How do the typical edge lengths behave?  These questions are natural under the interpretation of points as agents: given a societal choice about how much fairness is appropriate (i.e.\ a choice of $\gamma$), does an optimal solution exist?  If so, are there multiple solutions (leading to potential conflict)?  Are there solutions that treat all locations equally?  Can a solution be obtained by local procedures, without recourse to a central authority?  Such questions are of interest far more broadly, but matching provides a clean mathematical setting in which answers are already intricate and subtle.  Our results indicate that fairness tends to lead to difficulties: larger $\gamma$ means that the set of minimal matchings is less well-behaved.

Besides the $2$-colour case introduced above, we consider \df{$1$-colour} matching.  A perfect $1$-colour matching of $R$ is a set $M$ of unordered pairs in which each point of $R$ appears exactly once.  All the above definitions apply analogously to the $1$-colour case.

We introduce one further complication.  Unless stated otherwise we allow \df{partial} matchings, in which each point belongs to \emph{at most} one edge. Points that belong to no edge are \df{unmatched}.  A matching is \df{perfect} if it has no unmatched points.  We extend the definition of $\gamma$-minimality to a partial matching $M$ by declaring each unmatched point to have infinite cost (where infinities add to give strictly larger infinities in the manner of infinite ordinals).  More precisely, we insist that for any finite set of edges together with any finite set of unmatched points, the restriction of $M$ to the incident and unmatched points has as few unmatched points as possible, and then subject to that constraint, the total cost of the edges is minimized (as before).
% (equivalently, the unmatched points are all of the same colour, or in the $1$-colour case there is at most one unmatched point); then, subject to that constraint, the total cost is minimized as before.  (See \cref{defs} for more formal definitions.)
Our first main result is that in fact unmatched points never occur in dimension $d=1$.

\begin{thm}[Perfectness]\label{perfect}
 Fix any $\gamma\in[-\infty,\infty]\cup\{1-,1+\}$.
  Let $R$ be a Poisson process of intensity~$1$ on $\R$.  (Respectively, let $R$ and $B$ be independent Poisson processes of intensity~$1$ on $\R$.) Almost surely, every $\gamma$-minimal $1$-colour (respectively, $2$-colour) matching of $R$ (respectively, $R$ and $B$) is perfect.
\end{thm}

The conclusion of \cref{perfect} is stronger and more subtle than it might at first appear.  Indeed, for random $R$ (and $B$) we call a \emph{random} matching $M$ on a joint probability space with $R$ (respectively, and $B$) a \df{matching scheme}, and we say that it is \df{invariant} if the joint law of $(R,M)$ (respectively $(R,B,M)$) is invariant under the action of every translation of $\R^d$.  Standard arguments imply relatively easily that any invariant matching scheme is a.s.\ perfect, for any $d$.  However, $\gamma$-minimal matchings need not be invariant (as we shall see).  Proving \cref{perfect} requires ruling out all the uncountably many possible non-perfect matchings simultaneously, for almost every choice of $R$ and $B$.  Note that we do not know whether the same conclusion holds in $\R^d$ for $d\geq 2$.

Our next goal is to classify the set of all $\gamma$-minimal matchings
according to the value of $\gamma$.  We start with the $2$-colour case in
dimension $d=1$, where our picture is essentially complete.  It turns out
that there is a pronounced change in behavior at $\gamma=1$.  We also
address matching schemes, and besides invariance we distinguish the
following properties.  We say that a matching scheme is a
\df{factor} if it is invariant, and if $M$ can be expressed as a deterministic function of $R$ (respectively, of $(R,B)$).
%where the function commutes with the action of each translation of $\R^d$.
We say that an edge \df{crosses} a set $S$ if the closed line segment joining its endpoints intersects $S$.  We call a matching \df{locally finite} if every bounded set is crossed by only finitely many edges; otherwise it is \df{locally infinite}.

\begin{samepage}
\begin{thm}[$2$-colour classification]\label{2colclass}
  Let $R$ and $B$ be independent Poisson processes of intensity~$1$ on $\R$, and fix $\gamma$.  Almost surely, the set of all $\gamma$-minimal $2$-colour matchings of $R$ and $B$ is as follows.
  \begin{enumerate}[\rm (i)]
    \item Let $\gamma\in\{1+\}\cup(1,\infty]$.  The $\gamma$-minimal matchings form a countable family $(M^k)_{k\in\Z}$.  Each $M^k$ is locally finite.  There is no invariant $\gamma$-minimal matching scheme.
    \item Let $\gamma=1$. There are uncountably many $\gamma$-minimal matchings.  There are uncountably many factor $\gamma$-minimal matching schemes.
    \item Let $\gamma=1-$.  The $\gamma$-minimal matchings consist of a countable family $(M_k)_{k\in\Z}$ of locally finite matchings together with two locally infinite matchings $M_\infty$ and $M_{-\infty}$.  The only factor schemes are $M_\infty$ and $M_{-\infty}$, and the only invariant schemes are these two and mixtures %(with arbitrary weights)
        of them.
    \item Let $\gamma\in[-\infty,1)$.  There is exactly one $\gamma$-minimal matching $M$. It is locally infinite and a factor.
  \end{enumerate}
\end{thm}
\end{samepage}

We sometimes call the parameter regimes $\gamma>1$ \df{supercritical}, $\gamma\in\{1,1\pm\}$ \df{critical}, and $\gamma<1$ \df{subcritical}.

We next compare the matchings for different $\gamma$.  As suggested by \cref{1d}, for all $\gamma>1$ the sets of matchings are identical, while for $\gamma<1$ they are closely related. Given two matchings $M$ and $M'$ of the same set(s) of points, their union forms the edge-set of a multigraph.  If each of its components is finite then we say that $M$ and $M'$ have \df{finite differences}.

\begin{thm}[$2$-colour comparisons]\label{2colcomp}
    Let $R$ and $B$ be independent Poisson processes of intensity~$1$ on $\R$, and consider their $2$-colour matchings.
    \begin{enumerate}[\rm (i)]
      \item Almost surely, the set of $\gamma$-minimal matchings is identical to the set of $\gamma'$-minimal matchings for all $\gamma,\gamma'\in \{1+\}\cup(1,\infty]$.
      \item Let $\gamma\in[-\infty,1)$ and $\gamma'\in[-\infty,1)\cup\{ 1-\}$.  Almost surely, the $\gamma$-minimal matching and any $\gamma'$-minimal matching have finite differences.
    \end{enumerate}
\end{thm}

We next turn to quantitative questions: how efficient are the matchings, and how locally can they be determined?  The following standard concept allows us to consider typical points.  For an invariant perfect $2$-colour matching scheme $M$, let $(R^*,B^*,M^*)$ be the Palm process of $R$, with $(B,M)$ as background processes.  This can be interpreted as $(R,B,M)$ viewed from a typical red point, or conditioned to have a red point at $0$. (See \cref{defs} or \cite[Ch.~7]{kallenberg} or \cite{hpps} for details). We denote the associated Palm probability measure $\P^*$.  Analogous definitions apply in the $1$-colour case.  The \df{matching distance} $X$ of $M$ is the random distance from the origin to its partner under the Palm measure, i.e.\
$X:=|M^*(0)|$ where $M^*(0)$ is defined by $\e{0,M^*(0)}\in M^*$.
%(in other words, the distance from a typical red point to its partner).

Suppose that $M$ is a factor matching scheme.  We say that it is \df{finitary} if, under the Palm measure, the partner $M^*(0)$ of the origin can be determined by examining the restriction of the point processes $(R^*,B^*)$ (or $R^*$ in the $1$-colour case) to a ball $\{x\in\R^d:|x|< L\}$, where the \df{coding radius} $L$ is an almost surely finite random variable (itself a function of $(R^*,B^*)$, or $R^*$ respectively).  When this holds it is immediate that $X\leq L$ almost surely.

\begin{thm}[$2$-colour tail bounds]\label{2coltail}
    Let $R$ and $B$ be independent Poisson processes of intensity~$1$ on $\R$, and consider $2$-colour matchings.  Let $\gamma\in [-\infty,1)$ and let $M$ be the unique $\gamma$-minimal matching, or let $\gamma=1-$ and let $M$ be one of the two $\gamma$-minimal factor matchings $M_\infty,M_{-\infty}$.
\begin{enumerate}[\rm (i)]
  \item The matching distance $X$ satisfies
  $$\E^* X^{1/2}=\infty \qquad\text{but}\qquad \P^*(X>x)<cx^{-1/2},\quad x>0,$$
  for some $c=c(\gamma)>0$.
  \item The matching is a finitary factor of $(R,B)$, with coding radius $L$ satisfying
  $$\P^*(L>\ell)<C\ell^{-\alpha},\qquad \ell>0,$$
  for some $\alpha=\alpha(\gamma)>0$ and $C=C(\gamma)>0$.
\end{enumerate}
\end{thm}

Turning to the $1$-colour case, there is again a sharp change in behavior at $\gamma=1$, but some details are different.  The $\gamma>1$ regime features a matching scheme that is invariant but not a factor.  In fact we can make a slightly stronger statement:  assign i.i.d.\ labels, each uniform on $[0,1]$, to the points of $R$ (conditional on $R$).  We say that $M$ is a factor \df{of i.i.d.\ labels} if $M$ can be expressed as a deterministic function of $R$ and the labels, where the function commutes with translations.

\begin{thm}[$1$-colour classification]\label{1colclass}
  Let $R$ be a Poisson process of intensity $1$ on~$\R$.
  \begin{enumerate}[\rm (i)]
    \item Let $\gamma\in [1,\infty]\cup\{1-,1+\}$.  Almost surely there
exist exactly two $\gamma$-minimal matchings, which we denote
$M_{+},M_{-}$.  Both are locally finite.  The set of matchings is
identical for all such $\gamma$.  The only invariant
$\gamma$-minimal matching
scheme is an equal mixture of $M_{+},M_{-}$.  It is not a factor, nor a factor of i.i.d.\ labels.
    \item
Let $\gamma\in [-\infty,1)$.
There exists an invariant $\gamma$-minimal $1$-colour matching scheme.
  \end{enumerate}
\end{thm}

For $1$-colour matchings in the $\gamma<1$ regime we lack proofs of uniqueness and finite differences (although we expect that they hold), and we do not know whether there exist factor matching schemes.  An exception is $\gamma=-\infty$: stable matchings are almost surely unique and perfect (for all $d$ and for $1$ and $2$ colours), and much more is known about them -- see \cite{hpps}.  Turning to matching distance, we establish the following bounds.

\begin{samepage}
\begin{thm}[$1$-colour tail bounds]\label{1coltail}
    Let $R$ be a Poisson processes of intensity~$1$ on~$\R$.
\begin{enumerate}[\rm (i)]
    \item Let $\gamma\in [1,\infty]\cup\{1-,1+\}$ and let $M$ be the invariant $\gamma$-minimal matching scheme.  The matching distance $X$ satisfies
        $$\P^*(X>x)=e^{-x}, \qquad x>0.$$
    \item Let $\gamma\in [-\infty,1)$ and let $M$ be any invariant $\gamma$-minimal matching scheme.  We have
        %$$\E^* X=\infty \qquad\text{but}\qquad
        $$\P^*(X>x)<c x^{-1}, \quad x>0,$$
        for some $c=c(\gamma)>0$.
\end{enumerate}
\end{thm}
\end{samepage}

If it could be established that there were a unique $\gamma$-minimal $1$-colour matching for $\gamma\in(-\infty,1)$ (as in the $2$-colour case, and the case $\gamma=-\infty$) then it would automatically be a factor.  By \cite[Theorem~3(i)]{hpps} it would follow that $\E^* X=\infty$, complementing the bound in \cref{1coltail}(ii) and establishing a sharp change in tail behavior across $\gamma=1$.

In higher dimensions the picture is much less complete.  We establish existence of $\gamma$-minimal matchings, and even of invariant schemes, in many but not all cases.

\begin{thm}[Higher dimensions]\label{higher}
Let $R$ (and $B$) be (independent) Poisson process(es) of intensity~$1$ on $\R^d$ and let $\gamma\in [-\infty,\infty]$.  There exists an invariant perfect $\gamma$-minimal matching scheme in each of the following three cases:
 \begin{enumerate}[\rm (i)]
  \item $1$-colour, $d\geq 2$, $\gamma <\infty$;
  \item $2$-colour, $d=2$, $\gamma<1$;
  \item $2$-colour, $d\geq 3$, $\gamma<\infty$.
 \end{enumerate}
\end{thm}

It is far from clear whether any $\gamma$-minimal matching exists in the remaining cases.  The cases $\gamma=1$ and $\gamma=\infty$ of $2$-colour matching in $\R^2$ seem particularly interesting.  The former case was discussed in \cite{h}, where in particular it was shown that any such matching must be locally finite.

\subsection*{Related work}

Minimal matchings of finitely many random points on bounded regions have been studied extensively.  Much of the focus has been on asymptotic behaviour of the total cost, which of course can be bounded without considering the structure of the minimal matching itself.  See for example \cite{bobkov-ledoux,talagrand,akt,fmr,holden-peres-zhai}.  We mention in particular the recent remarkably precise analysis of $2$-minimal matchings in dimension $2$ in \cite{ambrosio-glaudo-trevisan,ambrosio-glaudo}.

Our focus on infinite point sets and the structure of the set of minimal matchings is, for the most part, new.  An exception is stable matchings (or $(-\infty)$-minimal matchings in our terminology).  Gale and Shapley \cite{gs} introduced stable matching of finite sets under general preference orders. Stable matching of infinite point processes was first considered in \cite{hp} and investigated further in \cite{hpps}.  Various extensions and applications appear in  \cite{deijfen-holroyd-martin,holroyd-martin-peres,deijfen-haggstrom-holroyd,
deijfen-holroyd-peres,deijfen-lopes,daley-last,hhp}.  Note that perfectness,
existence and uniqueness are straightforward to prove for the stable case
\cite{hpps}.  As we shall see, the subcritical regime $\gamma<1$ shares certain features with the stable case.  Some arguments from stable matching carry over relatively easily to the more general setting, but others apparently do not, providing a useful perspective on their limitations.

The closest approach to the questions considered here is the article \cite{h} by one of the current authors, where $1$-minimal matchings were considered in relation to the problem of non-crossing matchings in the plane.  A notion of minimality is also considered for the related problem of allocations between the Poisson process and Lebesgue measure in \cite{huesmann-sturm}.

\subsection*{Outline of the paper}

\cref{defs} contains detailed definitions and elementary results; particularly useful are uniqueness of minimal matchings of finite point sets (\cref{coincidence}) and classification of edge arrangements in dimension~$1$ (\cref{possibilities}).  In \cref{perf} we prove perfectness, \cref{perfect}, via a somewhat surprising alternating path argument.  \cref{super,lev} analyse supercritical and critical matchings in dimension $1$.  The arguments are combinatorial in nature and involve precise characterizations of the matchings.  A key tool is a random walk representation introduced in \cref{lev}.  In \cref{lim} we give two different limiting arguments for establishing existence of minimal matchings, proving \cref{higher} as well as parts of the earlier theorems.  Moreover this section introduces quasistability, a key property of the subcritical regime.  In \cref{uniq} we use quasistability and random walk properties to establish uniqueness (\cref{2colclass}(iv)) as well as finite differences (\cref{2colcomp}(ii)) and finitariness (\cref{2coltail}) in the $2$-colour subcritical regime.  \cref{orient} proves a further structural property of this regime, and \cref{tail} proves the remaining tail bounds of \cref{1coltail,2coltail}.

\section{Preliminaries}
\label{defs}

In this section we give full formal definitions and establish some basic facts, including the various side remarks made in the introduction.

\subsection{Matchings}
Let $R$ be a set, whose elements we call \df{points}.  A \df{$1$-colour matching} $M$ of $R$ is a set of unordered pairs of points, called \df{edges}, such that each point belongs to at most one edge.  Similarly, let $R$ and $B$ be two sets whose points we call \df{red} and \df{blue} respectively.  A \df{$2$-colour matching} of $R$ and $B$ is a set $M \subseteq R\times B$ of ordered pairs, called \df{edges}, such that each point belongs to at most one edge.  We denote an edge $e=\e{x,y}$, where $x,y$ are its endpoints.
If a point $x$ belongs to some edge of a matching $M$ then we call $x$ \df{matched} and write $M(x)$ for its \df{partner}, i.e.\ the unique other point belonging to the edge. Otherwise we write $M(x)=\infty$ and call $x$ \df{unmatched}.  The set of unmatched points is $M^{-1}(\infty)$.  A matching is \df{perfect} if it has no unmatched points.

\subsection{Minimal matchings of finite sets}

We are interested in matchings in $\R^d$ that minimize a cost function of the edge lengths; we start with finite point sets.
Let $f:(0,\infty)\to\R$ be a non-decreasing function, and let $|\cdot|$ denote the Euclidean norm. Let $R$ (\rs $R$ and $B$) be (\rs disjoint) finite subset(s) of $\R^d$, and let $M$ be a $1$-colour matching of $R$ (\rs a $2$-colour matching of $R$ and $B$).
We write
$$f[M]:=\sum_{\e{x,y}\in M} f\bigl(|x-y|\bigr).$$
We sometimes call $f$ a \df{cost function}, and $f[M]$ the \df{cost} of the matching.  %(Costs can be negative, in which case they can be interpreted as rewards).

Let $\prec$ denote the lexicographic order on real sequences: $a_1\cdots a_k\prec b_1\cdots b_k$ if and only if $a_i<b_i$ where $i$ is the smallest index for which $a_i\neq b_i$, while sequences of unequal length are compared by padding the shorter one with $-\infty$ entries at the end.  We say that $M$ is $f(\cdot)$\df{-minimal} if for every $1$-colour (\rs $2$-colour) matching $m$ of the same set(s) we have
$$\bigl(\#M^{-1}(\infty),f[M]\bigr)\preceq \bigl(\#m^{-1}(\infty),f[m]\bigr).$$
(So we first minimize unmatched points, and then cost).  Note that for any $f$, an $f(\cdot)$-minimal $1$-colour matching has 0 or 1 unmatched points according to the parity of $\#R$, while in the $2$-colour case there are $|\#R-\#B|$ unmatched points, all of the more numerous colour.
%Note also that for finite $R$ (and $B$) there are only finitely many matchings, so there is at least %one $f(\cdot)$-minimal matching.

We focus on power laws and logarithms.  For $\gamma\in \R$ define
\begin{equation}\label{f}
f(x)=f_\gamma(x)=\begin{cases}
  x^\gamma,&\gamma>0;\\
  \log x,&\gamma=0;\\
  -x^\gamma,&\gamma<0.
\end{cases}
\end{equation}
We abbreviate $f_\gamma(\cdot)$-minimal to \df{$\gamma$-minimal}.  When $\gamma$ is clear from context we sometimes simply say \df{minimal}.

We will show later in this section that the function $f=\log$ arises as the limit $\gamma\to 0$, justifying the notation $f_0$.  Similarly, the cases $\gamma=\pm\infty,1\pm$ defined next arise as the appropriate limits.

For a 1- or $2$-colour matching $M$ of finite set(s) we write $|M|_\uparrow$ and $|M|_\downarrow$ respectively for the increasing and decreasing orderings of the multiset of edge lengths $(|x-y|:\e{x,y}\in M)$.  We say that $M$ is \df{$(-\infty)$-minimal} if for any matching $m$ of the same set(s) we have
$$\bigl(\#M^{-1}(\infty),|M|_\uparrow\bigr)\preceq \bigl(\#m^{-1}(\infty),|m|_\uparrow\bigr),$$
(where  $(x,(y_1,y_2,\ldots))$ is interpreted as  $(x,y_1,y_2\ldots)$ for purposes of the order $\preceq$).
An \df{$\infty$-minimal} matching is defined in the same way but using the decreasing orderings $|\cdot|_\downarrow$.

Consider the special case of dimension $1$.  Let $e$ and $e'$ be two edges
of a matching in $\R$, with all four endpoints pairwise distinct.  We call
them \df{entwined} if exactly one endpoint of $e$ lies between the endpoints
of $e'$ (and hence vice versa).  We say that $e$ \df{straddles} $e'$ if both
endpoints of $e'$ lie between the endpoints of $e$.  If two edges neither
entwine nor straddle then we call them \df{separate}.
We call a matching $M$ of finite set(s) \df{$(1+)$-minimal} if it is
$1$-minimal and no two edges straddle, and \df{$(1-)$-minimal} if it is
$1$-minimal and no two edges are entwined.
Also,
in a $2$-colour matching, we say that an edge $\e{r,b}\in R\times B$ is
\text{oriented} \df{right} if $r<b$ and \df{left} if $b<r$.

\subsection{Infinite sets}

Now we extend the definitions to infinite sets.  Let $M$ be a $1$-colour (\rs
$2$-colour) matching of a countable set $R$ (\rs sets $R$ and $B$).  Call a
subset $R'\subseteq R$ (\rs a pair of
subsets $R'\subseteq R$ and $B'\subseteq B$) \df{compatible} with $M$ if all partners of points in $R'$ (\rs $R'\cup B'$) also belong to $R'$ (\rs $R'\cup B'$); in other words the subset(s) consist only of matched pairs and unmatched points.  In that case we write $M|_{R'}$ (\rs $M|_{R',B'}$) for the \df{restriction}: the set of edges whose endpoints belong to $R'$ (\rs $R'\cup B'$).

Here is our key definition.  For $\gamma\in[-\infty,\infty]\cup\{1-,1+\}$ we say that $M$ is \df{$\gamma$-minimal} if for any \emph{finite} compatible subset(s) $R'\subseteq R$ (and $B'\subseteq B$), the restriction $M|_{R'}$ (\rs $M|_{R',B'}$) is a $\gamma$-minimal matching of $R'$ (and $B'$).  Note that these definitions agree with the original ones when $R$ (and $B$) are finite.
Note moreover that the restriction of a $\gamma$-minimal matching to infinite compatible set(s) is $\gamma$-minimal.
We observe also that the definitions for $\gamma=1\pm$ extend in the expected way: a matching of infinite sets is $(1+)$-minimal if and only if it is $1$-minimal and no two edges straddle, and similarly for $1-$. We could extend the concept of $f(\cdot)$-minimality for a general function $f$ to infinite sets in the same way.

We next note an alternative characterization of $(-\infty)$-minimal matchings.  Recall that if $x$ is unmatched in $M$ then we write $M(x)=\infty$.  In that case we also write $|x-M(x)|:=\infty$.  We write $\vee$ and $\wedge$ for maximum and minimum respectively. We say that a 1- or $2$-colour matching $M$ is \df{stable} if for any two points $x,y$ (of opposite colours, in the $2$-colour case) we have
\begin{equation}|x-M(x)|\wedge|y-M(y)|\leq |x-y|.\label{stab}\end{equation}
The interpretation is that no two points would both prefer to be matched to each other over their current situations.  As remarked earlier, stable matchings have been studied extensively \cite{gs,hpps}.

\begin{samepage}
\begin{lemma}[Stability]\label{stable}
  Let $R$ (and $B$) be countable (disjoint) subset(s) of $\R^d$, and suppose that all distances between pairs of points (\rs of opposite colours) are distinct.  A $1$-colour (\rs $2$-colour) matching  is $(-\infty)$-minimal if and only if it is stable.
\end{lemma}
\end{samepage}

\begin{proof}
  Suppose that $M$ is not stable, so \eqref{stab} fails.  If $x$ and $y$ are both unmatched then we can match them to each other. If only $y$ (without loss of generality) is unmatched then we could match $x$ to $y$ instead of $M(x)$.  If both points are matched then we could match $x$ to $y$ and $M(x)$ to $M(y)$.  In each case, the modification shows that $M$ is not minimal.

  Suppose that $M$ is not minimal.  Consider a finite compatible set of points on which its restriction $m$ is not minimal, and let $m'$ be a minimal matching of the same points.  We can assume that every point either has different partners in $m$ and $m'$ or is unmatched in one but not the other (otherwise reduce the compatible set).  Minimality implies that the shortest edge $\e{x,y}$ of $m'$ is no longer than any edge of $m$, so by the distinct distances assumption it is strictly shorter than every edge of $m$.  But then $x,y$ violate \eqref{stab}, so $M$ is not stable.
\end{proof}

\subsection{Point processes}

We are interested in matching random sets.  Let $\red$ be a simple
point process on $\R^d$, which is formally a
locally finite random measure, where $\red(S)$
represents the number of points in $S\subseteq\R^d$.
We take $R$ to be its \df{support}, which is the random discrete
set of its points:
$$R=\supp\red:=\bigl\{x\in\R^d:\red(\{z\})=1\bigr\}.$$
%\SJx{We assume that $\red(S)<\infty$ for every bounded $S$; equivalently,
%that $R$ is discrete.}
Let $\blue$ be another simple point process, and let $B=\supp \blue$.  A $2$-colour \df{matching scheme} of $\red$ and $\blue$ is a simple point process $\mat$ on $(\R^d)^2$ such that a.s.\ $M:=\supp \mat$ is a matching of $R$ and $B$ (where $\red,\blue,\mat$ are assumed to be defined on some shared probability space).   Similarly, a $1$-colour \df{matching scheme} of $\red$ is a simple point process $\mat$ on the space of unordered pairs whose support is a $1$-colour matching of $R$ a.s.  Usually we suppress explicit mention of the random measures, and simply call $R$ and $B$ point processes, and $M$ a matching scheme.
%  The underlying random variables are still random measures, and all events of interest can be described in terms of them -- there is no need to introduce $\sigma$-algebras for random sets.

A translation $\theta$ of $\R^d$ acts on point sets via $\theta R=\{\theta x:x\in R\}$, and on matchings via $\theta M=\{\e{\theta x,\theta y}: \e{x,y}\in M\}$.
A point process $R$ is \df{invariant} if $R$ and $\theta R$ are equal in law
for each translation $\theta$ of $\R^d$.  A matching scheme $M$ is
\df{invariant} if $(R,M)$ (\rs $(R,B,M)$) is invariant in law under the
diagonal action $\theta(R,B,M)=(\theta R,\theta B,\theta M)$ of each
translation. A matching scheme $M$ is a \df{factor} if a.s.\ $M=F(R)$ (\rs
$M=F(R,B)$) for some measurable function $F$ that commutes with each translation of $\R^d$.  If $R$ (\rs $(R,B)$) is invariant then a factor matching scheme is invariant. (In fact, given only that $M=F(R)$ or $M=F(R,B)$ a.s., one can show that $F$ can be chosen to commute with translations if and only if $M$ is invariant).  A matching scheme $M$ is \df{ergodic} if $(R,M)$ (\rs $(R,B,M)$) is ergodic under the full group of translations of $\R^d$.
One can similarly consider invariance under isometries or other transformations, but we focus on translations.

We call a matching scheme \df{$\gamma$-minimal} if the matching is a.s.\ $\gamma$-minimal, and similarly for other properties, such as perfectness.
We emphasize two distinct viewpoints.  We can consider the random set of all
possible minimal matchings, as a function of the random sets $R$ and $B$.
Or we can consider a minimal matching scheme, which means that for almost
every choice of $R$ and $B$ we pick a minimal matching from that set
(perhaps using additional randomness, if the scheme is not a factor).  Matching schemes are mainly of interest when they are invariant.  Note that if there is a.s.\ a unique $\gamma$-minimal matching then automatically
there is an a.s.\ unique $\gamma$-minimal matching scheme (obtained by simply choosing the minimal matching as a function of the points).
Moreover if $R$ is invariant (\rs $(R,B)$ is jointly invariant) then this scheme is invariant and a factor.  (One can check that a matching scheme obtained in this way is indeed measurable in the appropriate sense by using \cite[Lemma~1.6]{kallenberg2} to re-express point processes as sums of point measures at random locations, together with the `selection theorem' \cite[Theorem~A.1.4]{kallenberg}.)

We focus on $1$-colour matchings of a homogeneous Poisson process $R$ of intensity $1$, or $2$-colour matchings of independent Poisson processes $R$ and $B$ of intensity $1$, on $\R^d$.

The following useful result says that minimal matchings are \emph{locally} unique, with the exception of the case $\gamma=1$ in dimension $d=1$ where a weaker statement holds.  A matching is \df{finitely supported} if only finitely many points are matched.

\begin{samepage}
\begin{prop}[Local uniqueness]\label{coincidence}
  Fix $d\geq 1$ and $\gamma\in\R$.  Let $R$ be a Poisson process of intensity $1$ on $\R^d$.  If $(d,\gamma)\neq(1,1)$ then almost surely there do not exist distinct finitely supported matchings $m$ and $m'$ of $R$ for which
  \begin{equation}\label{equal}f_\gamma[m]=f_\gamma[m'].\end{equation}
  If $(d,\gamma)=(1,1)$ then almost surely there do not exist finitely supported matchings $m$ and $m'$ with distinct matched sets that satisfy \eqref{equal}.
\end{prop}
\end{samepage}

\begin{proof}
  It suffices to consider matchings whose matched points lie within a fixed bounded set, and by scaling and conditioning we can take it to be the unit cube $[0,1]^d$.  Moreover, we can condition on the number of points in the cube, and consider two fixed matchings of them.  Therefore, let $x_1,\ldots,x_n$ be points in $[0,1]^d$, where $x_i=(x_{i,1},\ldots,x_{i,d})$.  Consider two fixed matchings of the set $\{1,\ldots,n\}$ and let $m$ and $m'$ be the corresponding matchings of $x_1,\ldots,x_n$.  Consider $\Delta:=f[m]-f[m']$ as a function of the positions of the points. It suffices to show that under the claimed conditions $\Delta\neq 0$ for a.e.\ choice of the variables $(x_{i,j})$ with respect to Lebesgue measure on $[0,1]^{dn}$.  We will do this using Fubini's theorem, by showing that $\Delta$ has a null set of zeros as a function of one variable, for almost all choices of the others.

  Firstly, suppose that the two matchings have distinct matched sets.  Without loss of generality suppose that $x_1$ is matched in $m$ but not in $m'$, and consider the dependence of $\Delta$ on the first coordinate $x_{1,1}$.  We have
  $$\Delta=f\Bigl(\sqrt{(x_{1,1}-a)^2+b^2}\Bigr)+c,$$
  where $a,b,c$ are functions of the other variables $(x_{i,j})_{(i,j)\neq (1,1)}$.
  Clearly for every choice of $a,b,c$ this has only finitely many zeros as a function of $x_{1,1}$.

  Secondly, suppose that the two matchings are distinct, and without loss of generality suppose that $x_1$ has different partners in $m$ and $m'$.  The dependence on $x_{1,1}$ is then of the form
  $$\Delta=f\Bigl(\sqrt{(x_{1,1}-a)^2+b^2}\Bigr)-f\Bigl(\sqrt{(x_{1,1}-a')^2+b'^2}\Bigr)+c,$$
  where $a,b,a',b',c$ are functions of the other variables.  This expression
  is piecewise real analytic
so it either has only finitely many zeros or it has non-trivial intervals of constancy.  Moreover, $a\neq a'$ unless the partners of $x_1$ in the two matchings have the same first coordinate, which happens only on a null set with respect to the other variables.  If $a\neq a'$ then $\Delta$ has intervals of constancy only if $\gamma=1$ and $b=b'=0$.  But for $d\geq 2$ the latter condition requires that two points have some coordinates equal, which again happens on a null set with respect to the other variables.
\end{proof}

%Combined with \cref{cont}, \cref{coincidence} shows that taking limits of $\gamma$ gives nothing new except in the cases $\gamma=-\infty,\infty,0$ discussed already, and the case $\gamma=1$ in $d=1$ to be discussed later.

Note that \cref{coincidence} immediately implies the analogous conclusion for $2$-colour matchings of independent Poisson processes $R$ and $B$, since any $2$-colour matching is a $1$-colour matching of the Poisson process $R\cup B$.

\begin{cor}[Distinct distances]\label{noneqd}
  Almost surely, all distances between pairs of points of a Poisson process are distinct.
\end{cor}

\begin{proof}
  Apply \cref{coincidence} to matchings consisting of only one edge.
\end{proof}

In particular, \cref{noneqd} shows that the assumption in \cref{stable} applies a.s.\ to Poisson processes.

For invariant minimal matching schemes, perfectness is easily established thanks to the following fact.

\begin{lemma}[Unmatched points]\label{unm}\
  \begin{enumerate}[\rm (i)]
   \item Let $M$ be any $1$-colour invariant matching scheme of an invariant point process $R$ on $\R^d$.  Almost surely, $M$ is either perfect or has infinitely many unmatched points.
   \item Let $M$ be any invariant $2$-colour matching scheme of jointly ergodic invariant point processes of equal intensity $R$ and $B$ on $\R^d$.  Almost surely, either $M$ is perfect or has unmatched points of both colours.
   \end{enumerate}
\end{lemma}

\begin{proof}[Proof of \cref{unm}]
In the $1$-colour case (i), suppose for some positive finite $k$ that there are exactly $k$ unmatched points with positive probability.  Conditional on this event, the process of unmatched points is still invariant and has exactly $k$ points.  Let $p$ be the (conditional) probability that it has at least one point in the  unit cube; then the expected number of points is $0$ if $p=0$ and $\infty$ if $p>0$, giving a contradiction.

In the $2$-colour case (ii),  consider the ergodic decomposition of $(R,B,M)$ with respect to the group of all translations of $\R^d$ \cite[Theorem~10.26]{kallenberg}.  Since $(R,B)$ is ergodic, in each ergodic component the processes of red and blue points have the same joint law as $(R,B)$ (otherwise we would have a non-trivial ergodic decomposition of $(R,B)$), and the matching is an invariant matching scheme.  Therefore we can assume without loss of generality that $(R,B,M)$ is ergodic.
If there are unmatched red points with positive probability then a.s.\ there are infinitely many, and they form an ergodic point process of positive intensity.  The same applies to blue points.  But by \cite[Proposition~7]{hpps} (a simple consequence of the mass transport principle), the processes of unmatched red and unmatched blue points have equal intensity.
\end{proof}

\begin{cor}[Invariant perfectness]\label{inv-perfect}
  Let $d\geq 1$ and let $R$ be an invariant point process (\rs let $(R,B)$
  be jointly ergodic invariant point processes of equal intensity) on $\R^d$
  and let $\gamma\in[-\infty,\infty]$.
  Any invariant $1$-colour (\rs $2$-colour) $\gamma$-minimal
  matching scheme is perfect.
\end{cor}

\begin{proof}
A $\gamma$-minimal $1$-colour matching can have at most one unmatched point, and a $\gamma$-minimal $2$-colour matching can have unmatched points of at most one colour.  Now apply \cref{unm}.
\end{proof}

For $2$-colour matchings in dimension~$1$, the following result from \cite{h} will be useful.  Note that for discrete sets $R,B\in \R$ (for instance, Poisson processes), each bounded interval contains only finitely many points.  Therefore a matching is local infinite if and only if some $x\in \R$ is crossed by infinitely many edges, which in turn is equivalent to the condition that every $x\in \R$ is.

\begin{prop}[Local infiniteness]\label{locinf}
  Let $R$ and $B$ be independent Poisson processes of intensity $1$ on $\R$.  Any invariant perfect $2$-colour matching scheme of $R$ and $B$ is a.s.\ locally infinite.
\end{prop}

\begin{proof}
  This follows from \cite[Theorem~3(i)]{h} together with
  ergodic decomposition \cite[Theorem~10.26]{kallenberg}.
\end{proof}

For an invariant $2$-colour matching scheme $M$ of $R$ and $B$, the Palm
process $(R^*,B^*,M^*)$ may be characterized as follows.  Let $\theta^x$
denote the translation by $x\in \R^d$.  Then for any non-negative measurable
map $h$ on the appropriate space,
$$\E \sum_{r\in R\cap [0,1)^d} h\bigl(\theta^{-r}(R,B,M)\bigr) = \lambda \,\E^* h\bigl(R^*,B^*,M^*\bigr) ,$$
where $\lambda$ is the intensity of the point process $R$.
If $R$ and $B$ are independent Poisson processes then the joint law of the Palm versions of the point processes themselves can be obtained by simply adding a red point at the origin: $(R^*,B^*)\eqd(R\cup\{0\},B)$.  Similar remarks apply to the $1$-colour case.  For more details see \cite[Section~2]{hpps} or \cite[Ch.~11]{kallenberg}.

We give a more detailed definition of finitary factors.  Let the matching scheme $M$ be a factor of $R$ and $B$.  Under the Palm measure, there is a map $H$ such that $M^*(0)=H(R^*,B^*)$ a.s.  Suppose that $H$ can be chosen, together with another map $L$ to $[0,\infty]$, in such a way that for any deterministic sets $r,b$ and any $r',b'$ that agree with them on the ball $\{x\in\R^d: |x|\leq L(r,b)\}$, we have $H(r,b)=H(r',b')$.  If in addition $L=L(R^*,B^*)<\infty$ a.s.\ then $M$ is a \df{finitary} factor of $(R,B)$ with \df{coding radius} $L$.

\subsection{Dimension one}

\newcommand{\f}[1]
{\begin{tikzpicture}[scale=.75,thick]
  \tikzstyle{every node}=[circle, draw, thin,
                        inner sep=0pt, minimum width=2mm]
  \useasboundingbox (-.5,-1) rectangle (3.5,1);
  #1
  \end{tikzpicture} }
\newcommand{\ff}[1]
{\begin{tikzpicture}[scale=.75,thick]
  \tikzstyle{every node}=[circle, draw, thin,
                        inner sep=0pt, minimum width=2mm]
  \useasboundingbox (-.5,-1.3) rectangle (3.5,1);
  #1
  \end{tikzpicture} }
\newcommand{\rrrr}{
\node[fill=red](a) at (0,0){};
\node[fill=red](b) at (.8,0){};
\node[fill=red](c) at (1.8,0){};
\node[fill=red](d) at (3,0){};
}
\newcommand{\rbbr}{
\node[fill=red](a) at (0,0){};
\node[fill=cyan!50!white](b) at (.8,0){};
\node[fill=cyan!50!white](c) at (1.8,0){};
\node[fill=red](d) at (3,0){};
}
\newcommand{\rbrb}{
\node[fill=red](a) at (0,0){};
\node[fill=cyan!50!white](b) at (.8,0){};
\node[fill=red](c) at (1.8,0){};
\node[fill=cyan!50!white](d) at (3,0){};
}
\newcommand{\rrbb}{
\node[fill=red](a) at (0,0){};
\node[fill=red](b) at (.8,0){};
\node[fill=cyan!50!white](c) at (1.8,0){};
\node[fill=cyan!50!white](d) at (3,0){};
}
\newcommand{\sep}{
\draw (a) to [bend left=60] (b);
\draw (c) to [bend left=60] (d);
}
\newcommand{\ent}{
\draw (a) to [bend left=60] (c);
\draw (b) to [bend left=60] (d);
}
\newcommand{\str}{
\draw[densely dashed] (b) to [bend right=60] (c);
\draw[densely dashed] (a) to [bend right=60] (d);
}
\newcommand{\sstr}{
\draw (b) to [bend right=60] (c);
\draw (a) to [bend right=60] (d);
}
\begin{figure}
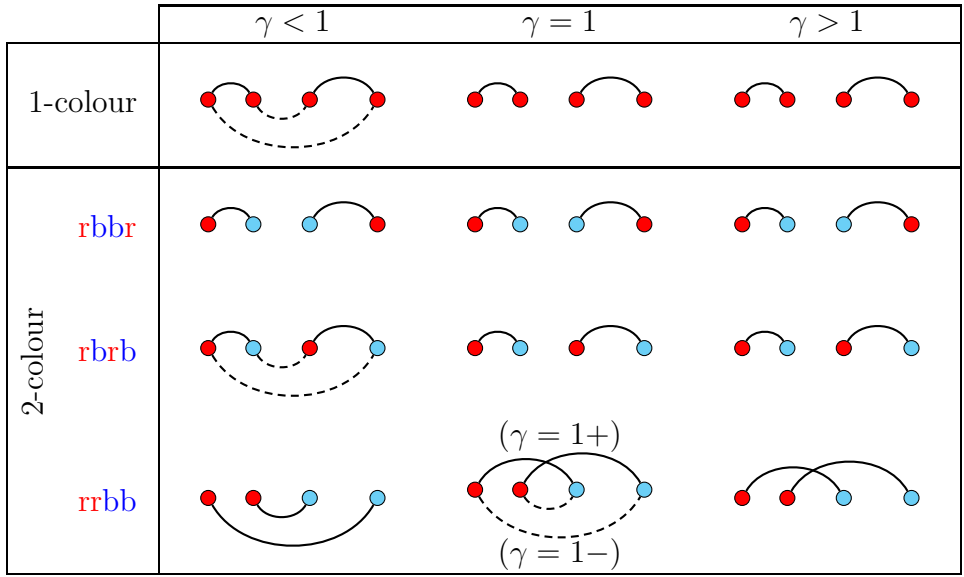

\centering
\begin{tabular}{|>{\centering\arraybackslash} m{.3cm}
>{\centering\arraybackslash} m{1cm}|
>{\centering\arraybackslash} m{3.2cm}>{\centering\arraybackslash} m{3.2cm}>{\centering\arraybackslash} m{3.2cm}|}
  \cline{3-5}

   \multicolumn{2}{c|}{}& $\gamma<1$ & $\gamma=1$ & $\gamma>1$ \\ \hline

 \multicolumn{2}{|c|}{$1$-colour} & \f{\rrrr\sep\str} & \f{\rrrr\sep} & \f{\rrrr\sep} \\ \hline

 & \textcolor{red}{r}\textcolor{blue}{bb}\textcolor{red}{r} & \f{\rbbr\sep} & \f{\rbbr\sep} & \f{\rbbr\sep} \\

 \rotatebox[origin=c]{90}{\;\;$2$-colour\!} & \textcolor{red}{r}\textcolor{blue}{b}\textcolor{red}{r}\textcolor{blue}{b} &
 \f{\rbrb\sep\str} & \f{\rbrb\sep} & \f{\rbrb\sep} \\

 & \textcolor{red}{rr}\textcolor{blue}{bb} & \f{\rrbb\sstr} &

 \ff{\rrbb\ent\str
 \node[rectangle,fill=none,draw=none] at (1.5,.95){$(\gamma=1+)$};
 \node[rectangle,fill=none,draw=none] at (1.5,-1.15){$(\gamma=1-)$};
 }
 & \f{\rrbb\ent} \\
 \hline
\end{tabular}
  \caption{Possible arrangements (separate, nested, or entwined) of two edges in a $\gamma$-minimal matching on $\R$, in the $1$-colour case (top row) and $2$-colour case (bottom three rows, according to the ordering of the colours).  The figures indicate only the order of the points, not their distances.  Solid lines and dashed lines indicate two different possible matchings;  in the bottom row with $\gamma=1$ both possibilities are $1$-minimal (with the tie broken as indicated for $\gamma=1\pm$), while in the other two cases the minimal choice depends on the distances between points.} \label{possibilities}
\end{figure}
The simple inequalities below are of central importance to the analysis of $\gamma$-minimal matchings in dimension $d=1$. Consider four points in $\R$ with successive distances $a,b,c$ between neighbouring pairs from left to right.  The following enables us to compare the cost of entwined versus straddling edges.

\begin{lemma}\label{ineq}
  Let $a,b,c>0$ and $\gamma\in\R$, and let $f=f_\gamma$ be as in \eqref{f}. We have
  \begin{align*}
  f(a+b+c)+f(b)&>f(a+b)+f(b+c), &\text{if }&\gamma>1;\\
  f(a+b+c)+f(b)&<f(a+b)+f(b+c), &\text{if }&\gamma<1.
  \end{align*}
\end{lemma}

\begin{proof}
  The quantity $f(t+c)-f(t)$ is strictly increasing in $t>0$ if $\gamma>1$ and strictly decreasing if $\gamma<1$; take $t=a+b$ and $t=a$.
\end{proof}

Obviously if $\gamma=1$ (so that $f$ is the identity) then
\begin{equation}f(a+b+c)+f(b)=f(a+b)+f(b+c).\label{equality}\end{equation}
For all $\gamma$, since $f_\gamma$ is strictly increasing it is also obvious that
\begin{equation}f(a)+f(c)<f(a+b)+f(b+c).\label{trivial}\end{equation}

Any two edges of a $\gamma$-minimal matching in $\R$ must be either
entwined, straddling, or separate.  In the $2$-colour case, one of the three
arrangements is disallowed, because we cannot match red-red and blue-blue.
For $\gamma\in(1,\infty)$, the ordering of the three costs is fixed:
separate $<$ entwined $<$ straddling, while the second inequality becomes an
equality for $\gamma=1$.  For $\gamma\in(-\infty,1)$,
entwined has the highest cost,
with the relative costs of the other two arrangements depending on the distances.  Consequently, for each choice of $\gamma$ and the colour-ordering of the points (rbbr, rbrb, rrbb, and the analogous sequences obtained by reversing the colours) there are either one or two possibilities, which we summarize in \cref{possibilities}.  Moreover, it is easy to identify the possibilities for $\gamma=1-$ and $\gamma=1+$, and to check that the possibilities for $\gamma=-\infty$ and $\gamma=\infty$ are identical to those for $\gamma\in(-\infty,1)$ and $\gamma\in(1,\infty)$ respectively.

\subsection{Limiting cases}

Next we justify the naming of the cases $\gamma=0,\pm\infty,1\pm$ by showing that they arise as appropriate limits of $\gamma$ for minimal matchings of finite sets, subject to certain regularity conditions.  These conditions vary slightly according to $\gamma$.

\begin{samepage}
\begin{prop}[Limits of $\gamma$]\label{cont-gamma}
  Let $d\geq 1$ and consider $1$- or $2$-colour matchings of fixed finite set(s) $R(,B)\subset \R^d$.
  \begin{enumerate}[\rm (i)]
    \item Fix $\gamma\in\R$ and suppose that there is a unique $\gamma$-minimal matching $M$.  Then for all $\mu\in\R$ sufficiently close to $\gamma$, $M$ is the unique $\mu$-minimal matching.
    \item Suppose that all pairs of points (of opposite colours in the $2$-colour case) have distinct distances. There is a unique $\infty$-minimal matching $M$, and for all $\mu$ sufficiently large, $M$ is the unique $\mu$-minimal matching.  The same statements hold for $(-\infty)$-minimal matching and $\mu$ sufficiently negative.
    \item Let $d=1$ and suppose that there is a unique $(1+)$-minimal matching $M$.  Then for all $\mu>1$ sufficiently close to $1$, $M$ is the unique $\mu$-minimal matching.  The same applies to $(1-)$-minimal matchings and $\mu<1$.
  \end{enumerate}
\end{prop}
\end{samepage}

To apply \cref{cont-gamma} we must verify its assumptions.  \cref{coincidence} gives conditions for uniqueness of minimal matchings, while the following gives existence.

\begin{lemma}[Finite existence]\label{fin-exist}
  Let $d\geq 1$ and $\gamma\in[-\infty,\infty]$, or let $d=1$ and $\gamma\in\{1+,1-\}$.  For any finite $R(,B)\subset \R^d$ there exists a $\gamma$-minimal $1$-colour matching or $R$ (\rs $2$-colour matching of $R$ and $B$).
\end{lemma}

In particular, for finite subsets of Poisson processes, \cref{cont-gamma}(i)
with $\gamma=0$ combined with \cref{coincidence,fin-exist} justifies the
notation $0$-minimal for $f=\log$, and also shows there are no further
limiting cases to be considered besides those under discussion here.
\cref{noneqd} similarly provides the distinct distances condition for
\cref{cont-gamma}(ii).
Verifying the assumptions of
\cref{cont-gamma}(iii)
is a little more delicate.  \cref{fin-exist} gives existence of
$(1\pm)$-minimal matchings.  For uniqueness, \cref{coincidence} shows that the set of matched points is uniquely determined, so we can restrict attention to perfect matchings of that set. Uniqueness then follows from a more detailed analysis of the various cases,
which can be found in the proofs of
\cref{2colclass}(i,iii) and \cref{1colclass}(i) later in this article.
\cref{cont-gamma} is not needed for these proofs, nor for any other results
of the article.

The proofs of \cref{cont-gamma,fin-exist} both use the next simple fact.

\begin{lemma}\label{cont}
   Let $g$ and $g_\eta$ be non-decreasing functions with $g_\eta\to g$ pointwise as $\eta\downarrow 0$, and consider $1$- or $2$-colour matchings of fixed finite set(s) $R(,B)\subset \R^d$.  For all sufficiently small $\eta$, every $g_\eta(\cdot)$-minimal matching is also $g(\cdot)$-minimal.
\end{lemma}

\begin{proof}
  Every $g(\cdot)$-minimal matching $M$ has the same number of unmatched points $u$ and the same cost $g[M]:=c$, say.
  Let $\delta>0$ be such that for every other matching $m$ with $u$ unmatched points,
  $$g[m]-c>\delta.$$
  Let $\ell_1,\dots,\ell_n$ be the distances between all pairs of points of $R\cup B$ (with multiplicities), and take $\eta$ sufficiently small that
  $$
   \sum_i \bigl|g_\eta(\ell_i)-g(\ell_i)\bigr|<\delta/2.
  $$
  This ensures that $|g_\eta[m]-g[m]|<\delta/2$ for every matching $m$.  We deduce that for every $g(\cdot)$-minimal matching $M$ and non-$g(\cdot)$-minimal matching $m$ with $u$ unmatched points we have
  $g_\eta[m]-g_\eta[M]>\delta-2\delta/2>0$.  This implies that $m$ cannot be $g_\eta(\cdot)$-minimal,
  as required.
\end{proof}

\begin{proof}[Proof of \cref{fin-exist}]
  The cases $\gamma\in [-\infty,\infty]$ are trivial: there are only finitely many matchings, so at least one of them must be minimal.  For $d=1$ and $\gamma=1+$, by the previous case, for every $\mu>1$ there exists a $\mu$-minimal matching.  By \cref{cont}, for $\mu>1$ sufficiently close to $1$, any such matching $M$ is also $1$-minimal.
  But by \cref{ineq} (see also \cref{possibilities}),
such an $M$ has no straddling edges, so it is $(1+)$-minimal.  An analogous argument applies for $\gamma=1-$.
\end{proof}

\begin{proof}[Proof of \cref{cont-gamma}]
 \sloppypar For $\gamma\neq 0$ the result of (i) follows immediately from \cref{cont}, since $f_\mu\to f_\gamma$ as $\mu\to\gamma$ and there is at least one $\mu$-minimal matching.  For $\gamma=0$, note that applying an increasing affine transformation to a function $f$ does not change the notion of $f(\cdot)$-minimality, so for $\mu\neq 0$ we can replace $f_{\mu}$ with the function $g_\mu(x)=(x^\mu-1)/\mu$, and observe that $g_{\mu}(x)\to \log x$ as $\mu\to 0$.

 Turning to (ii), existence of $(\pm\infty)$-minimal matchings follows from \cref{fin-exist}.
 Uniqueness follows from the distinct distances assumption because $|m|_\uparrow$ and $|m'|_\uparrow$ differ for distinct matchings $m\neq m'$, and similarly for $|\cdot|_\downarrow$.

 Let $\ell_1<\cdots<\ell_n$ be the ordered distances between all pairs of points (of opposite colours in the $2$-colour case).  For the $\infty$-minimal case it suffices to take $\mu$ large enough that $\ell_k^\mu > \ell_1^\mu+\cdots+\ell_{k-1}^\mu$ for all $k$, which is achieved if $n^{1/\mu}<\min_k \ell_k/\ell_{k-1}$.  Similarly for the $(-\infty)$-minimal case we take $\mu$ negative enough that $-\ell_k^\mu < -\ell_{k+1}^\mu -\cdots -\ell_n^\mu$ for all $k$.

  Finally, for (iii), suppose $M$ is the unique $(1+)$-minimal matching.  By
  \cref{cont}, for $\mu>1$ sufficiently close to $1$, every $\mu$-minimal
  matching is 1-minimal.  But by \cref{ineq}
 (see also \cref{possibilities}),
  any such matching has no straddling edges, so it must be $M$.  The argument for $1-$ is analogous.
\end{proof}

\subsection{Scale invariance}

The functions $f_\gamma$ have a scale-invariance property: for any $s>0$, the expression $f_\gamma(sx)$ can be written as an increasing affine function of $f_\gamma(x)$.  But applying an increasing affine map to the cost function does not change the minimal matchings.
Therefore, the set of $\gamma$-minimal $1$-colour matchings of a scaled set $sR$ is precisely the set of scaled matchings $sM$, for $M$ a $\gamma$-minimal matching of $R$, and similarly for the $2$-colour case.  The next result shows that essentially no other functions have this property, justifying our focus on  $f_\gamma$.  Again, this result is not needed for the proofs of the main theorems.

\begin{samepage}
\begin{prop}[Scale invariance]\label{scale}
  Let $f:(0,\infty)\to\R$ be continuously differentiable, non-decreasing and not constant. Suppose that for every finite set $R\subset \R$ and every $s>0$, we have that $M$ is an $f(\cdot)$-minimal $1$-colour matching of $R$ if and only if $sM$ is an $f(\cdot)$-minimal $1$-colour matching of $sR$.  Then there exist $a,b,\gamma\in\R$ such that
  $$f(x)=\begin{cases}
    ax^\gamma+b,& \gamma\neq 0,\\
    a\log x +b,&\gamma=0.
  \end{cases}
  $$
  %We have $a>0$ when $\gamma\geq 0$, and $a<0$ when $\gamma<0$.
\end{prop}
\end{samepage}

%We consider one-colour matchings on $\bbR$.
%We use partial matchings of finite sets; all matchings are supposed to leave
%at most one unmatched point.
%
%\begin{remark}
%  It is not enough to study complete matchings on an even number of points.
%In fact, unless I am mistaken, every convex $f$ then yields the same
%minimum, viz.\ $P_1P_2, P_3P_4, \dots$
%\end{remark}
%
%\begin{remark}
%  I do not define scale-invariant here. When there are ties, I assume that
%  the set of all minimal matchings is scale-invariant. (This is convenient,
%  but I think that weaker versions would be enough.)
%
%I also omit some other definitions of notation, since you in any case will
%edit and rewrite this proof.
%\end{remark}
%
%\begin{theorem}\label{T1}
%  Let $f:\qooo\to\oooo$ be increasing and continuously differentiable,
%and suppose that
%minimal  (partial) matchings are scale-invariant.
%Then there exist $p,a,b\in\oooo$ such that
%\begin{align}\label{t1}
%  f(x)=
%  \begin{cases}
%    a x^p + b, & p\neq 0,
%\\
%a \log x + b, & p=0.
%  \end{cases}
%\end{align}
%Furthermore,
%except for the trivial case $a=0$ when $f$ is constant,
%we have $a>0$ when $p\ge0$ and $a<0$ when $p<0$.
%\end{theorem}
%The converse is obvious. Moreover, the values of $a\neq0$ and $b$ do not
%matter for the minimal matchings, so we may assume $a=\pm1$ and $b=0$.

We break the proof into several lemmas.
First, since $f$ is not constant, note that $f'(t)>0$ for some $t>0$.
Replacing $f(x)$ by $f(tx)$, we may assume (for notational convenience)
that $f'(1)>0$.
In particular, this implies $f(2)>f(1)$. Choose $\delta>0$ such that
\begin{align}
  \label{hamlet}
2f(1+\delta) < f(2)+f(1),
\end{align}
and, furthermore,
\begin{align}\label{laertes}
  f'(x)>0, \qquad x\in[1,1+\delta].
\end{align}
Fix this $\delta$ for the remainder of the argument.

\begin{lemma}
  \label{L1}
Suppose that $x,y,z,w\in[1,1+\delta]$ and $s>0$. Then
\begin{align}
  \label{ophelia}
f(x)+f(y) \le f(z)+f(w)
\implies
f(sx)+f(sy) \le f(sz)+f(sw).
\end{align}
\end{lemma}

\begin{proof}
Assume
\begin{align}\label{polonius}
f(x)+f(y) \le f(z)+f(w) .
\end{align}
By symmetry, we may assume  $x\ge y$ and $z\ge w$.
Consider the set of 5 points $R=\{r_1,\dots,r_5\}$ with successive gaps, in order, $x,w,y,z$.
Any minimal matching must have exactly one unmatched point.
Consider first the matchings consisting only of nearest neighbours; there
are three such matchings, with costs $f(x)+f(y)$, $f(x)+f(z)$ and $f(w)+f(z)$.
If $z<y$, then $w\le z<y\le x$, which contradicts \eqref{polonius}.
Hence, $z\ge y$, which together with \eqref{polonius} shows that the
matching $M:=\{\e{r_1,r_2},\e{r_3,r_4}\}$
with cost $f(x)+f(y)$ is minimal among these three.

Furthermore, the cost of this matching is at most $2f(1+\delta)$,
while every matching including a non-neighbour pair costs at least
$f(2)+f(1)$; hence
\eqref{hamlet} shows that $M$ is a minimal matching.

By the scale-invariance assumption,
$sM$ is a minimal matching of $sR$.
In particular, $f(sx)+f(sy)\le f(sz)+f(sw)$.
\end{proof}

\begin{lemma}\label{L2}
  If $x,z\in[1,1+\delta]$ and $s>0$, then
\begin{align}
    \label{fortinbras}
\frac{f'(sx)}{f'(x)} = \frac{f'(sz)}{f'(z)}.
 \end{align}
\end{lemma}

\begin{proof}
The  assumption \eqref{laertes} implies that $f:[1,1+\delta]\to[f(1),f(1+\delta)]$
has a differentiable inverse $g:[f(1),f(1+\delta)]\to[1,1+\delta]$,
with
\begin{align}\label{gertrude}
  g'(f(y))=\frac{1}{f'(y)},
\qquad y\in[1,1+\delta].
\end{align}

By continuity, it suffices to consider $x,y\in[1,1+\delta)$.
%\SJm{This is a bit silly, since $\delta$ is arbitrarily small in any
%  case. But if I don't do this (as in the previous version), I would have to
%  replace $\delta$ by $\delta/2$ several times later, which seems more
%  inconvenient. (This was an error in my previous version.) What do you think?}
Let $\eps\ge0$ be so small that $f(x)+\eps, f(z)+\eps<f(1+\delta)$.
Define
$$
  w:=g\bigpar{f(x)+\eps},\qquad\qquad
y:=g\bigpar{f(z)+\eps}.
$$
Then
\begin{align}
  f(x)+f(y)=f(x)+f(z)+\eps=f(z)+f(w).
\end{align}
Hence, \cref{L1} (twice) yields $f(sx)+f(sy)=f(sz)+f(sw)$;
in other words,
\begin{align}\label{rosencrantz}
  f(sx)+f\bigpar{s g\bigpar{f(z)+\eps}}
=
  f(sz)+f\bigpar{s g\bigpar{f(x)+\eps}}.
\end{align}
Since \eqref{rosencrantz} holds for every small $\eps\ge0$, we may take the
(right) derivative at $\eps=0$ and obtain, by the chain rule,
recalling $g(f(z))=z$ and $g(f(x))=x$,
\begin{align}\label{guildenstern}
  f'(sz) s g'(f(z)) = f'(sx)sg'(f(x)),
\end{align}
which yields \eqref{fortinbras} by \eqref{gertrude}.
\end{proof}

\begin{lemma}\label{L3}
  $f'(x)\ne0$ for all $x>0$.
\end{lemma}
\begin{proof}
  Suppose that $f'(x)=0$ for some $x>1$, and let $x_0$ be the infimum of all
  such $x$; by continuity, $f'(x_0)=0$ and thus $x_0>1+\delta$.
Let $x:=1$ and $z:=1+\delta$, and take $s:=x_0/z>1$. Then $f'(sz)=f'(x_0)=0$,
and thus \cref{L2} shows that $f'(s)=f'(sx)=0$. This is a contradiction,
because $1<s<sz=x_0$.

Similarly, $f'(x)=0$ for some $x<1$ also leads to a contradiction.
\end{proof}

\begin{lemma}
  \label{L4}
For any $x,y,t>0$,
\begin{align}
  \label{macbeth}
\frac{f'(tx)}{f'(x)}=\frac{f'(ty)}{f'(y)}.
\end{align}
\end{lemma}
\begin{proof}
Assume first that $0<x\le y\le(1+\delta)x$. Then let $z:=y/x\in[1,1+\delta]$.
Apply \cref{L2} to $1$ and $z$, with $s:=x$ and $s:=tx$; this yields
\begin{align*}
\frac{f'(x)}{f'(1)} = \frac{f'(y)}{f'(z)};\qquad
\frac{f'(tx)}{f'(1)} = \frac{f'(ty)}{f'(z)}.
\end{align*}
Dividing, we obtain \eqref{macbeth}
in the case $1\le y/x\le 1+\delta$.

By induction on $n$, we see that \eqref{macbeth} holds when $1\le
y/x\le(1+\delta)^n$ for every $n\ge1$, and thus whenever $x\le y$.
By symmetry, \eqref{macbeth} holds for all $x,y>0$.
\end{proof}

\begin{proof}[Proof of \cref{scale}]
Taking $y=1$ in \eqref{macbeth} shows that
\begin{align}
  \frac{f'(tx)}{f'(1)}
= \frac{f'(t)}{f'(1)}\cdot   \frac{f'(x)}{f'(1)}.
\end{align}
In other words, $x\mapsto f'(x)/f'(1)\in (0,\infty)$
is multiplicative and continuous, and
thus
there exists a real number $\rho$ such that
\begin{align}
  \frac{f'(x)}{f'(1)}  = x^\rho,
\end{align}
i.e., with $c:=f'(1)>0$,
\begin{align}
  f'(x)=cx^\rho.
\end{align}
This yields the claimed expressions with $\gamma=\rho+1$.
\end{proof}

\section{Perfectness}
\label{perf}

In this section we prove \cref{perfect} which states that minimal matchings are perfect in dimension $1$.  The key step is \cref{shorter} below, which holds in all dimensions.

Fix $d$ and $\gamma$, and consider $1$-colour (respectively, $2$-colour) matchings of set(s) of points $R$ (and $B$).  We say that a point $x$ is \df{potentially unmatched} if there exists a $\gamma$-minimal matching of $R$ (and $B$) in which it is unmatched.

\begin{prop}\label{shorter}
  Fix $d\geq 1$ and $\gamma\in (-\infty,\infty]$, and let $R$ (and $B$) be (independent) Poisson process(es) of intensity $1$ on $\R^d$.  Consider $\gamma$-minimal $1$-colour matchings of $R$ (respectively, $2$-colour matchings of $R$ and $B$).  Almost surely, if $x$ and $y$ are any two potentially unmatched points (respectively, of opposite colours) there exists an infinite sequence of distinct points $x_0,x_1,x_2,\ldots\in R(\cup B)$ (respectively, of alternating colours), with $x_0=x$, such that
  $$|x_i-x_{i+1}|\leq |x-y|\quad \forall i\geq 0.$$
\end{prop}

\begin{proof}
  Let $x$ and $y$ be potentially unmatched, and let $M$ and $N$ be $\gamma$-minimal matchings in which each of them is unmatched, respectively.  Since a $1$-colour minimal matching can have at most one unmatched point, and a $2$-colour minimal matching cannot have unmatched points of both colours, $x$ is matched in $N$ and $y$ is matched in $M$.
  Let $G$ be the multigraph with vertex set $R(\cup B)$ and whose edges are the edges of $M$ and $N$.  Each vertex has degree at most $2$, and $x$ and $y$ each have degree $1$.  Let $H$ be the component of $G$ containing $x$, which must be a finite or infinite path starting from $x$, with edges alternately in $N$ and $M$.  In the $2$-colour case, the colours of the points alternate also.

  Suppose that $H$ is finite.  The vertex set $V$ of $H$ is compatible with both $M$ and $N$, so the restrictions of $M$ and $N$ are $\gamma$-minimal matchings of $V$, with different matched sets.  This contradicts \cref{coincidence} if $\gamma<\infty$, or \cref{noneqd} if $\gamma=\infty$.

  So $H$ is infinite.  Let $x=x_0,x_1,x_2,\ldots$ be its vertices in order along the path. First suppose $\gamma<\infty$.  Let $c_0=f_\gamma(|x_0-y|)$ and $c_i=f_\gamma(|x_{i}-x_{i-1}|)$ for $i\geq 1$.  For $k\geq 0$,
  consider modifying $N$ by switching matched and non-matched edges along the alternating sequence $y,x_0,x_1,\ldots,x_{2k+1}$, so that $x_{2k+1}$ becomes unmatched instead of $y$.  Since $N$ is $\gamma$-minimal, this modification cannot decrease cost, so we conclude
  \begin{align*}
  c_0+c_2+\cdots+c_{2k}&\geq c_1+c_3+\cdots+c_{2k+1}.\\
  \intertext{Similarly, switching $M$ along $x_0,x_1,\ldots,x_{2k}$ gives}
  c_1+c_3+\cdots+c_{2k-1}&\geq c_2+c_4+\cdots+c_{2k}.
  \end{align*}
  Adding the two inequalities and cancelling the repeated terms gives $c_0\geq c_{2k+1}$.  Doing the same but using the $k+1$ case of the second inequality gives $c_0\geq c_{2k+2}$.  Since $f_\gamma$ is nondecreasing we obtain the claimed bound for both odd and even indices.

  The argument for $\gamma=\infty$ is similar.  Let $\ell_0=|x_0-y|$ and $\ell_i=|x_{i}-x_{i-1}|$ for $i\geq 1$.
  The same modifications as above yield
  \begin{align*}
  \ell_0\vee\ell_2\vee\cdots\vee\ell_{2k}&\geq \ell_1\vee\ell_3\vee\cdots\vee\ell_{2k+1},\\
  \ell_1\vee\ell_3\vee\cdots\vee\ell_{2k-1}&\geq \ell_2\vee\ell_4\vee\cdots\vee\ell_{2k},
  \end{align*}
  for $k\geq 0$.  Therefore, for all $j\geq 0$,
  $$\ell_0\vee\cdots\vee\ell_j\geq \ell_1\vee\cdots\vee\ell_{j+1},$$
  and we conclude by induction on $j$ that $\ell_0\geq \ell_j$ for all $j$.
\end{proof}

\begin{proof}[Proof of \cref{perfect}]
  In the case $\gamma=-\infty$ (stable matching; \cref{stable}), a.s.\ perfectness is known to hold in all dimensions; see e.g.\ \cite[Proposition~9]{hpps}.

  Let $\gamma\in(-\infty,\infty]$.  Suppose for a contradiction that there
  exist non-perfect $\gamma$-minimal matchings with positive probability.
  The set of all potentially unmatched points forms an ergodic invariant
  point process.  Since it is not a.s.\ empty it has infinitely many points
  a.s.  In $2$-colour case, the same argument applies to each colour, so
  by colour symmetry a.s.\ there are potentially unmatched points of both colours.  Therefore, by \cref{shorter}, there exists an infinite sequence of points $x_0,x_1,\ldots$ of $R(\cup B)$ with the distances $|x_i-x_{i+1}|$ bounded above (by a random but a.s.\ finite quantity).  However, this is impossible for a Poisson process in $\R$.  (Indeed, a.s.\ for every positive integer $n$, every point lies between two intervals of length at least $n$ that contain no points).

  Finally, the cases $\gamma=1+,1-$ follow trivially from $\gamma=1$.
\end{proof}

%We remark that the same proof yields perfectness for Poisson processes on the strip $\R\times [0,1]$.

\section{The supercritical case}
\label{super}

In this section we analyse the $\gamma$-minimal matchings for $\gamma>1$ (together with some of $\gamma=1,1\pm$, depending on the number of colours) in $d=1$.  These cases are relatively straightforward, and permit simple explicit descriptions of the matchings.  We start with the $1$-colour case.

\begin{proof}[Proof of \cref{1colclass}(i)]
  Let $\gamma\in[1,\infty]\cup \{1+,1-\}$ and consider a $1$-colour $\gamma$-minimal matching of any set $R\subset \R$.
  By \cref{ineq} and \eqref{equality}, \eqref{trivial} (see also \cref{possibilities}), any two edges of such a matching must be separate.

  First suppose that $R$ has finite even cardinality, say $R=\{x_1,\ldots,x_{2n}\}$ where $x_1<\cdots<x_{2n}$.  By \cref{fin-exist} there is a minimal matching, and it must be perfect.  The only perfect matching with pairwise separate edges is
  $$m=\bigl\{\e{x_1,x_2},\e{x_3,x_4},\ldots,\e{x_{2n-1},x_{2n}}\bigr\},$$
  therefore this is the unique minimal matching.

  Now let $R$ be a Poisson process. By \cref{perfect} a.s.\ every
  minimal matching is perfect.  Let $\cdots<x_{-1}<x_0<x_1<\cdots$ be the
  points of $R$ in order, indexed so that $x_{-1}<0<x_0$ say.
%\SJm{I assume, here and at similar places,
%that you deliberately tacitly ignore the null event when this
%  is not possible (because $0\in R$). This is fine with me.}
There are exactly two perfect matchings with pairwise separate edges:
  \begin{align*}
    M_+:&=\{\ldots,\e{x_{-1},x_0},\e{x_1,x_2},\e{x_3,x_4},\ldots\},\\
    M_-:&=\{\ldots,\e{x_{-2},x_{-1}},\e{x_0,x_1},\e{x_2,x_3},\ldots\}.
  \end{align*}
  We call these two the \df{alternating} matchings. (They may be defined
  for any discrete set $\{x_i:i\in\Z\}\subset\R$ unbounded in both directions).
  These are therefore the only possible candidates for minimal matchings.
  Both are indeed minimal, because any restriction to a finite compatible set gives a finite matching of the form of $m$ above.

  We turn to matching schemes.  Any minimal matching scheme $M$ must concentrate on the alternating matchings $\{M_+,M_-\}$.  Such a scheme is characterised by the conditional probability
  $$\phi=\phi(R)=\P(M=M_+\mid R),$$
  and indeed $\phi$ can be chosen to be any measurable map from point configurations to $[0,1]$.
  %=\P(0 \text{ is crossed by some edge of }M\mid R).$$
  It is easy to check that taking $\phi\equiv \tfrac12$ gives an invariant matching scheme $M$.  This scheme amounts to flipping a fair coin independently of $R$, and choosing $M_+$ or $M_-$ according to the outcome.  Equivalently, if we write $X_t=\ind[t \text{ is crossed by an edge}]$ then $(X_t)_{t\in\R}$ is the stationary continuous-time Markov chain with state space $\{0,1\}$ and transition rate $1$ between the states in each direction.

  The scheme just defined is not a factor, and no other choice of the function $\phi$ gives an invariant scheme.  These unsurprising but slightly delicate facts are proved in \cite[Lemma~11 and the following Remark]{hpps}.  Briefly, the argument is as follows.  First, no alternating factor matching exists.  This is proved using local approximations of events and mixing properties of the Poisson process.  Second, from any invariant scheme other than the one with $\phi\equiv \tfrac12$ one could construct an alternating factor matching, a contradiction.  Finally, the mixing argument from \cite{hpps} extends immediately to show that no alternating matching can be expressed as a factor of i.i.d.\ labels.
\end{proof}

Now we turn to the $2$-colour case.

\begin{proof}[Proof of \cref{2colclass}(i) and \cref{2colcomp}(i)]
   Let $\gamma\in(1,\infty]\cup \{1+\}$ and consider any $2$-colour $\gamma$-minimal matching of sets $R,B\subset \R$.
  By \cref{ineq} and \eqref{equality}, \eqref{trivial}, (see also \cref{possibilities}), two edges cannot straddle, and they can only be entwined if they have the same orientation.  Hence, for any two red points $r<r'$ and two blue points $b<b'$, the matching cannot contain both $\e{r,b'}$ and $\e{r',b}$: the matching respects order.

  Suppose that $R$ and $B$ are of equal finite cardinality, say $R=\{r_1,\ldots,r_n\}$ and $B=\{b_1,\ldots,b_n\}$ where $r_1<\cdots<r_n$ and $b_1<\cdots<b_n$.  Minimal matchings exist by \cref{fin-exist}, and must be perfect. Hence, by the above remark, the unique minimal matching is
  $$m=\{\e{r_1,b_1},\ldots,\e{r_n,b_n}\}.$$

  Now let $R$ and $B$ be independent Poisson processes of intensity $1$.  By \cref{perfect}, a.s.\ every minimal matching is perfect.  Let the red points be $\cdots<r_{-1}<r_0<r_1<\cdots$ and the blue points $\cdots<b_{-1}<b_0<b_1<\cdots$, where $r_{-1}<0<r_0$ and $b_{-1}<0<b_0$.  By the remark in the first paragraph, the matching must respect the orderings.  That is, any minimal matching is of the form
  $$M^k:=\bigl\{\e{r_{i+k},b_{i}}:i\in\Z\bigr\}$$
  for some $k\in \Z$.  Moreover, each $M^k$ is indeed minimal, because any restriction to a finite compatible set gives a finite matching of the form of $m$ above.

  The number of edges of $M^k$ that cross $0$ is exactly $|k|$.  Hence the number of edges that cross a bounded interval containing $0$ is at most $|k|$ plus the number of points in the interval, which is finite.  Hence $M^k$ is locally finite.

  Any invariant minimal matching scheme would therefore be locally finite, and this contradicts \cref{locinf}, so there is no such scheme.
\end{proof}

\section{Levels and critical cases}
\label{lev}

Next we address $\gamma\in\{1-,1\}$ in the $2$-colour case.  We start by introducing a simple tool that will be important for $\gamma<1$ as well.

Given disjoint discrete
sets $R,B\subset\R$ we define the associated \df{walk} $W=W_{R,B}:\R\to\Z$ by
\begin{equation}
\begin{aligned}
  W(0-)&=0;\\
  W(y)-W(x)&=\#\bigl(R\cap(x,y]\bigr)-\#\bigl(B\cap(x,y]\bigr),\quad x<y.
\end{aligned}\label{walk}
\end{equation}
The walk takes a step up at a red point and down at a blue point.  The choice of starting point $W(0-)=0$ will be convenient for later technical steps involving the Palm process.  For the current discussions we can (optionally) assume that there is no point at $0$ so that $W(0)=0$. If $R$
and $B$ are independent Poisson processes of intensity $1$ then $W$ is a
continuous-time simple symmetric random walk.
More precisely, the jump times of $W$ form a Poisson process of
  intensity 2, and $W$ considered at these times is a symmetric simple
  random walk; thus $W$ is a particular case of a compound Poisson process.

For $k\in\Z$, define \df{level} $k$ to be the set $\Lambda_k=\Lambda_k(R,B)$ of points where the walk moves between $k$ and $k+1$:
\begin{equation}\label{level}
\Lambda_k:=\Bigl\{x\in \R:\{W(x-),W(x+)\}=\{k,k+1\}\Bigr\}.
\end{equation}
Note that the levels $(\Lambda_k)_{k\in \Z}$ form a partition of $R\cup B$, and that the elements of a level alternate in colour.  %In the Poisson process setting, a.s.\ each level is unbounded above and below.

\begin{remark}
  Walks and levels will be most useful for $\gamma\leq 1$, but they also provide a convenient description of the matchings $(M^k)_{k\in\Z^d}$ introduced in \cref{super} in the context of $\gamma> 1$.  In the matching $M^k$,
  some elements of $\R$ are not crossed by any edge: these form precisely the interior of the set $\{x:W(x)=k\}$.  The closure of $\{x:W(x)>k\}$ is a disjoint union of bounded intervals (corresponding to excursions of $W$ above $k$) containing the points of $\bigcup_{j\geq k} \Lambda_j$.  Each such interval contains equal numbers of red and blue points.  Within each such interval, the $i$th red point in the interval is matched to the $i$th blue point, via a right oriented edge.  Similarly, each interval of the closure of $\{x:W(x)<k\}$ contains points in $\bigcup_{j<k} \Lambda_j$, with the $i$th red point to the $i$th blue point via a left oriented edge.
\end{remark}

Now we focus on matchings with no entwined edges, as holds for $\gamma$-minimal matchings with $\gamma<1$ and $\gamma=1-$.

\begin{lemma}[Levels]\label{levels}
  Let $R,B\subset\R$ be disjoint discrete
sets, and define levels as above.  Let $M$ be any perfect $2$-colour matching of $R$ and $B$ in which no two edges are entwined.  Any point and its partner belong to the same level.
\end{lemma}

\begin{proof}
  Suppose on the contrary that $r$ and $b$ are partners that belong to different levels, where $r<b$ say.  Then $W(b-)-W(r+)\neq 0$, so the numbers of red and blue points in the open interval $I=(r,b)$ are unequal.  So, since the matching is perfect, some point $x$ in $I$ must have its partner outside $I$.  But then  $\e{r,b}$ and $\e{x,M(x)}$ are entwined.
\end{proof}

Now we address $(1-)$-minimal $2$-colour matchings.  We can explicitly characterize all the matchings.

\begin{proof}[Proof of \cref{2colclass}(iii)]
\sloppypar   First consider any $(1-)$-minimal $2$-colour matching of any disjoint discrete sets $R,B\subset \R$.  By \eqref{equality}, \eqref{trivial} and the definition of $(1-)$-minimality (see also \cref{possibilities}), no two edges are entwined, and no two edges of opposite orientations straddle.
%Suppose moreover that the matching is perfect.

Consider any matching with the properties:
\begin{equation}\label{properties}
  \text{perfect, no entwined edges, no straddling edges of opposite orientations.}
\end{equation}
By \cref{levels}, each point is in the same level as its partner.
We claim that within a given level, no two edges can straddle.  Indeed, suppose $\e{r,b}$ is a right-oriented edge (without loss of generality) that straddles another edge in the same level.  Let $b'$ be the first point in the level to the right of $r$.  Then $b'$ is blue and $b'\neq b$, but its partner $r'$ must be to its right to avoid entwining, so the two straddling edges $\e{r,b}$ and $\e{r',b'}$ have opposite orientations, a contradiction which establishes the claim.

   Now let $R$ and $B$ be of equal finite cardinality.
   %There exist minimal $(1-)$-minimal matchings by \cref{fin-exist}, and they must be perfect.
     We claim that there is a unique $(1-)$-minimal matching, and that it is the unique matching with the properties \eqref{properties}.
   To establish this, note that every minimal matching has these properties, and there is at least one minimal matching by \cref{fin-exist}.  On the other hand, let $M$ be any perfect matching with the properties.  By the claim above, there are no straddling edges within a level.  Therefore, there is only one possible matching of each level: $\{\e{x_1,x_2},\ldots,\e{x_{2n-1},x_{2n}}\}$ where $x_1<\cdots<x_{2n}$ are the points of the level.  Hence there is at most one matching with the given properties, completing the proof of the claims.

   Now let $R,B$ be independent Poisson processes of intensity $1$.  Let $M$ be any
   $(1-)$-minimal matching (if one exists).  By \cref{perfect}, a.s.\ $M$ is perfect, so by the previous discussions, the matching is confined to levels and has no entwined edges or straddling edges within levels.
   By the recurrence of the random walk $W$, each level $\Lambda_k$ is unbounded in both directions.  Therefore, restricted to a given level $\Lambda_k$, there are two possible matchings -- the alternating matchings defined in the proof of \cref{1colclass}(i).  One has all edges oriented left, and the other all right.  We denote them $m_k^-$ and $m_k^+$ respectively.

   We now address the relationship between levels.  For each $k\in\Z$, a.s.\ there exist two red points $r<r'$ with no other points between them and with $r\in\Lambda_k$, and hence $r'\in\Lambda_{k+1}$.  Suppose that the matching at level $k$ is $m_k^+$, so that $r$ is matched to the right.  Then $r'$ must also be matched to the right, otherwise the edges would be entwined.  Therefore the matching at level $k+1$ must be $m_{k+1}^+$.
   Consequently, there must exist some $k\in\Z\cup \{-\infty,\infty\}$ such that the matching $M$ takes the form
   \begin{equation}\label{defmk}
   M_k:=\Bigl(\bigcup_{j< k-1/2} m_j^- \Bigr)
   \cup \Bigl(\bigcup_{j> k-1/2} m_j^+\Bigr).
   \end{equation}
   (Here the first union in empty if $k=-\infty$, and the last is empty if $k=\infty$.)

   We need to check that each of the matchings $M_k$ defined above is indeed
   minimal.  We claim that they each have the properties in
   \eqref{properties}.  Once this is established, the same obviously holds
   for any finite subset of the edges, and therefore by the earlier
   discussion, every restriction to a finite compatible set is minimal, so
   the matching $M_k$ is minimal.  To prove the claim, consider two edges
   $\e{r,b}$ and $\e{r',b'}$ at respective levels $j$ and $j'$.  If $k-\tfrac12$
   does not lie between $j,j'$ then without loss of generality suppose
   $k-\tfrac12<j\leq  j'$.   Then $r<b$, and $W>j$ throughout the interval
   $(r,b)$.  Suppose that one of $r',b'$, say $r'$ without loss of
   generality, lies between $r$ and $b$.  Since $W$ makes a step from $j'$
   to $j'+1$ at $r'$, it must step back down to $j'$ between $r'$ and $b$.
   So $\e{r,b}$ straddles $\e{r',b'}$ and both are oriented right.  Now
   suppose that $k-\tfrac12$ lies between $j$ and $j'$, say
   $j<k-\tfrac12<j'$.  Then $W<k$ on the interval $(b,r)$ and $W>k$ on the
   interval $(r',b')$, so the two edges are separate.
Thus $M_k$ is minimal. Consequently, a.s.\
the minimal matchings are precisely $(M_k:k\in\Z\cup \{-\infty,\infty\})$.

   Next we address local finiteness.  The points
of a level alternate in colour, and in $M_k$, the intervals between them alternate between being crossed by no edge and one edge of that level.  If $j>k-\frac12$ then $x\in\R\setminus(R\cup B)$ is crossed by an edge of level $j$ if and only if $W(x)>j$; for $j<k-\frac12$ the condition becomes $W(x)\leq j$.  Therefore the total number of edges of $M_k$ that cross $x\in\R\setminus(R\cup B)$ is $|W(x)-k|$.  In particular, $M_{-\infty}$ and $M_\infty$ are locally infinite, while $(M_k:k\in\Z)$ are all locally finite.

   Now suppose that $M$ is an invariant minimal matching scheme.  By \cref{locinf}, $M$ is locally infinite a.s., so it must concentrate on $\{M_{-\infty},M_\infty\}$.  On the other hand, $M_{-\infty}$ can be described as follows: the partner of a red point $r$ is
   \begin{equation}\label{meshalkin}
   M_{-\infty}(r)=\min\Bigl\{b>r: \#\bigl(R\cap[r,b]\bigr)=\#\bigl(B\cap[r,b]\bigr)\Bigr\}.\end{equation}
   And $M_\infty$ has a similar characterization with the colours reversed.  Therefore both $M_{-\infty}$ and $M_\infty$ are invariant, and indeed they are factors, and hence ergodic.  For a general invariant minimal matching scheme $M$, the event $A=\{M=M_{-\infty}\}$ is a.s.\ equivalent to the event that every edge is oriented right, which is translation-invariant.  Therefore if $M$ is ergodic then $A$ has probability $0$ or $1$.  Hence there are no further ergodic matching schemes besides $M_{-\infty}$ and $M_\infty$, and in particular no other factors.  Finally, by ergodic decomposition \cite[Theorem~10.26]{kallenberg}, every invariant minimal matching scheme $M$ is a mixture of $M_{-\infty}$ and $M_\infty$ (that is, $M=M_{\eta}$ where $\eta$ is independent of $(R,B)$ and takes value $\infty$ with probability $p$ and $-\infty$ otherwise, where $p\in[0,1]$ is arbitrary).
%
%   By \cref{levels}, each point is in the same level as its partner.  %Since $\#R=\#B$ we have $W(-\infty)=W(\infty)$, and therefore each level has an even number of points.
%  The points of a level alternate in colour.  We claim that within a given level, no two edges can straddle.  Indeed, suppose $(r,b)$ is a right-oriented edge (without loss of generality) that straddles another in the same level.  Let $b'$ be the first point in the level to the right of $r$.  Then $b'$ is blue and $b'\neq b$, but its partner is to the right, so the straddling edges have opposite orientations, establishing the last claim.  Therefore, there is only one possible matching of the points of a level: $\{(x_1,x_2),\ldots,(x_{2n-1},x_{2n})\}$ where $x_1<\cdots<x_{2n}$.  Hence there is at most one matching with the given properties, proving the uniqueness claims.
%
%  Now let $R$ and $B$ be independent Poisson processes of intensity $1$.
\end{proof}

\begin{proof}[Proof of \cref{2colclass}(ii)]
  We address $1$-minimal $2$-colour matchings of Poisson processes.  The
  $(1-)$-minimal matching $M_{-\infty}$ from the proof of part (iii) above
  is by definition $1$-minimal.  We will modify it.  Let $r<r'<b'<b$ be $4$
  consecutive points, with $r,r'\in R$ and $b,b'\in B$ and no other points
  between $r$ and $b$.  A.s.\ there are infinitely many such $4$-tuples, and
  no two of them overlap with each other.  The matching $M_{-\infty}$
  matches the pairs $\e{r,b},\e{r',b'}$.  Matching $\e{r,b'},\e{r',b}$
  instead gives another $1$-minimal matching.
% (by \eqref{equal}).
(It suffices to consider finite compatible subsets containing
  $r,r',b,b'$, and then the two induced matchings have the same cost.)
Therefore, we get uncountably many minimal matchings by performing this modification at an arbitrary set of $4$-tuples.  Moreover, if $E$ is a measurable subset of $[0,\infty)$ then we can perform the modification at those tuples for which $|r'-b'|\in E$.  Each choice of $E$ gives a minimal factor matching scheme, and varying $E$ (for example over sets of the form $E_I:=\bigcup_{i\in I} [i,i+1)$ for $I\subseteq\N$) gives uncountably many such schemes.
\end{proof}

\section{Quasistability, and existence via limits}
\label{lim}

In this section we present two different limiting arguments that establish existence of minimal matchings, and indeed of invariant matching schemes, in certain cases.  The first involves a limit of minimal matchings on large finite boxes.  The second, an extension of methods of \cite{h}, involves matchings whose average cost approaches the infimum.  In both cases the key step is to establish that no points are `matched to infinity' in the limit.  In the first case this relies on a property of the subcritical regime $\gamma<1$ which we call \emph{quasistability}.  This property will be important for later proofs also.  The second case relies on a uniform bound on the average cost.  Recall that invariant minimal matching schemes are perfect by \cref{inv-perfect}.

\begin{prop}[Existence via quasistability]\label{quasilimit}
  \sloppypar Let $\gamma\in (-\infty,1)$ and $d\geq 1$, and consider
  (independent) Poisson process(es) $R$ (and $B$) of intensity $1$ on $\R^d$.  There
  exists an invariant perfect
$\gamma$-minimal $1$-colour (\rs $2$-colour) matching
  scheme.
\end{prop}

\begin{prop}[Existence via finite average cost]\label{minlimit}
  Let $\gamma\in(0,\infty)$ and $d\geq 1$, and consider $1$-colour
  (respectively $2$-colour) matchings of (independent) Poisson process(es)
  $R$ (and $B$) of intensity $1$ on $\R^d$.  Suppose that there exists some
  invariant perfect
 matching scheme whose matching distance $X$ satisfies $\E^*
  X^\gamma<\infty$.  Then there exists an invariant perfect
$\gamma$-minimal
  matching scheme.
\end{prop}

   Quasistability is analogous to stability \eqref{stab}, but with an extra multiplicative constant.  Recall that if $x$ is an unmatched point of $M$ then we write $M(x)=\infty$ and $|x-M(x)|=\infty$.

\begin{prop}[Quasistability]\label{quasi}
  For each $\gamma\in[-\infty,1)$ there exists
  $\kappa=\kappa(\gamma)\in[1,\infty)$ with the following property.  For any
  $d\geq 1$, if $M$ is a $\gamma$-minimal $1$-colour matching
  of a set $R\subset\R^d$
  and $x$ and $y$ are two distinct points
  in $R$, then
  \begin{equation}\label{kappa-ineq}
    |x-M(x)|\wedge|y-M(y)|\leq \kappa\,|x-y|.
  \end{equation}
  The same holds for a $\gamma$-minimal $2$-colour matching provided $x$ and $y$ have different colours.
\end{prop}

\begin{proof}
 The points $x$ and $y$ cannot be both unmatched in a minimal matching.  Moreover, if only $y$ (say) is unmatched then, since $\kappa\geq 1$, if \eqref{kappa-ineq} fails then $|x-y|<|x-M(x)|$, so matching $x$ to $y$ instead of $M(x)$ would reduce the cost.  Therefore we can assume $x$ and $y$ are both matched.
  We can also assume that they are not matched to each other,
  otherwise \eqref{kappa-ineq} holds trivially.

 Suppose that \eqref{kappa-ineq} fails, i.e.\
 \begin{equation}|x-M(x)|,\;|y-M(y)|> \kappa|x-y|.\label{fails}\end{equation}
 We claim that, with the appropriate choice of $\kappa$, modifying $M$ by matching instead the pairs $\e{x,y}$ and $\e{M(x),M(y)}$ strictly reduces the cost, in contradiction to minimality.

 For $\gamma=-\infty$ we can take $\kappa=1$ and the claim is immediate.  For $-\infty<\gamma<0$ we can take any $\kappa> 2^{-1/\gamma}$, for then
 \begin{align*}
   |x-M(x)|^\gamma+|y-M(y)|^\gamma&\leq 2 \kappa^\gamma |x-y|^\gamma\\
   <|x-y|^\gamma&\leq |x-y|^\gamma+|M(x)-M(y)|^\gamma.
 \end{align*}

 For the remaining cases $0\leq \gamma<1$ we write
 $$u=\frac{|x-M(x)|}{|x-y|}; \qquad v=\frac{|y-M(y)|}{|x-y|}$$
 so that the assumption \eqref{fails} becomes $u, v>\kappa$.  By the triangle inequality,
 $$|M(x)-M(y)|\leq |x-y|\,(1+u+v).$$
 Therefore, in the case $\gamma=0$, the change in cost associated with the modification is
 $$\log \frac{|x-y|\cdot|M(x)-M(y)|}{|x-M(x)|\cdot|y-M(y)}\leq \log\frac{1+u+v}{uv}=
 \log\Bigl(\frac{1}{uv}+\frac{1}{u}+\frac{1}{v}\Bigr).$$
 We can take $\kappa=3$, so that this is at most $\log(\tfrac 19+\tfrac 23)<0$.

 Finally, for $0<\gamma<1$, let
 $$g(u,v)=u^\gamma+v^\gamma-1-(1+u+v)^\gamma,$$
 so that the reduction in cost is at least $|x-y|\,g(u,v)$.  By differentiating we see that $g$ is increasing in $u$ and $v$, so we need only show that $g(u,u)>0$ for some $u$.  But in fact we have
 \[g(u,u)=u^\gamma\,\bigl[2-u^{-\gamma}-(1/u+2)^\gamma\bigr]\sim u^\gamma\,(2-2^\gamma)\to\infty
 \]
 as $u\to\infty$.
\end{proof}

\begin{proof}[Proof of \cref{quasilimit}]
  We address the $2$-colour case first.  Fix $d\geq 1$ and $\gamma\in(-\infty,1)$.  We will construct the desired matching scheme as a limit, and for this we interpret point processes and matchings as random measures.  Let $\leb$ denote Lebesgue measure on $\R^d$.

  Let $n$ be a positive integer.
  Let $\red_n$ and $\blue_n$ be independent Poisson processes of intensity $1$ on $\R^d$.  Let $Q_n$ be uniformly distributed on the cube $[0,n)^d$ and independent of $(\red_n,\blue_n)$.  Define a \df{$n$-tile} to be any set of the form $[0,n)^d+z+Q_n$ for $z\in\Z^d$.  The $n$-tiles form a random partition of $\R^d$.  By \cref{coincidence,fin-exist}, within each $n$-tile there is a.s.\ a unique $\gamma$-minimal $2$-colour matching of the points of $\red_n$ and $\blue_n$ that lie in the tile.  Let $\mat_n$ be the point process on $(\R^d)^2$ whose support is the union over all tiles of these matchings.  Then $\mat_n$ is an invariant matching scheme of $\red_n$ and $\blue_n$.

  Let $\X_n=(\red_n,\blue_n,\mat_n)$, which we interpret as a point process
  on the disjoint union $\R^d\sqcup\R^d\sqcup(\R^d)^2$.  We claim that the
  sequence $(\X_n)_{n=1}^\infty$ is relatively compact in distribution with
  respect to the vague topology on
%simple
point measures on this space.
This follows from \cite[Lemma~16.15]{kallenberg}.  Indeed,
for any bounded $A\subset\R^d$ we have $\red_n(A)\eqd \blue_n(A) \eqd\mbox{Poi}(\leb
A)$, while
any bounded $A\subset(\R^d)^2$ is a subset of $U\times \R^d$ for some Borel bounded $U$, and $\mat_n(U\times \R^d)\leq \red_n(U)\eqd \mbox{Poi}(\leb U)$; therefore $(\mat_n(A))$ is a tight sequence.  Hence there is a subsequence $(n(k))$ and a point process
  $\X=(\red,\blue,\mat)$ for which
  \begin{equation}\X_{n(k)}\tod \X\quad\text{as }k\to\infty\label{conv}\end{equation}
  in the aforementioned topology.

  %The above convergence implies that for any continuous non-negative compactly supported function %$g$ we have $\int g\, d\X_{n(k)}\tod \int g\,d\X$ by \cite[Lemma~6.16(i)]{kallenberg}.  Hence %$\red$ and $\blue$ are independent Poisson processes of intensity $1$.  By comparing such a %function with its image $g\circ \theta$ under the diagonal action of a translation $\theta$ we %deduce that the process $\X$ inherits the translation-invariance of $\X_n$.

  %Let $\A$ be the family of bounded Borel subsets of $\R^d$ with $\leb$-null boundary.
%  Note that if $D\subset \R^d$ is any $\leb$-null set then $\mat_n(D\times\R^d)=0$ a.s.\ and so the same holds for $\mat$.  Therefore by \cite[Lemma~6.16(i)]{kallenberg}, if $U,V\subset\R^d$ are bounded Borel sets with $\leb$-null boundaries then $\mat_{n(k)}(U\times V)\tod\mat(U\times V)$. Also $\red_{n(k)}(U)\tod \red(U)$ and similarly for $\blue$.  Moreover,
%  these convergence in distribution statements hold jointly for any finite collection of such sets.
%
%  We next check that $\mat$ is a matching scheme of $\red$ and $\blue$.  This follows because $\mat_n(U\times V)\leq \red_n(U)\wedge \blue_n(V)$ a.s.\ for any bounded $U,V$ with null boundary, so the same holds in the limit.
%
%  We will show that $\mat$ is perfect, and then that it is minimal.  There are two possible sources of unmatched points: those that were unmatched within a tile (because of a colour inbalance), and those that are ``matched to infinity'' in the limit.

  Let $\A$ be the set of all bounded Borel subsets of $\R^d$ with $\leb$-null boundaries (which includes balls and rectangles).
  Note that
any $\leb$-null set $D\subset \R^d$
is contained in an open set $D'$ with $\leb D'$ arbitrarily small, and
$\mat_n(D'\times\R^d)\eqd\mbox{Poi}(\leb D')$ for all $n$, which implies
that $\mat(D\times\R^d)=0$ a.s.; similarly $\mat(\R^d\times D)=0$ a.s.
Therefore by
 \cite[Lemma~16.16]{kallenberg}, if $U,V\in\A$ then $\mat_{n(k)}(U\times
 V)\tod\mat(U\times V)$. Also $\red_{n(k)}(U)\tod \red(U)$ for $U\in\A$, and
 similarly for $\blue$.  Moreover, these convergence in distribution
 statements hold jointly for any finite collection of such sets.

  Applying the above to $\red(U_i)$ for a family of disjoint $U_i\in\A$ and similarly for $\blue$ we deduce that $\red$ and $\blue$ are independent Poisson processes of intensity $1$.  Comparing sets in $\A$ with their translations shows that $\X$ inherits the translation invariance of $\X_n$.
  For any $U,V\in \A$ we have $$\mat_n(U\times V)\leq \red_n(U)\wedge \blue_n(V) \quad\text{a.s.},$$ so the same holds in the limit.  Hence, $\mat$ is an invariant matching scheme of $\red$ and $\blue$.

  Next we show that $\mat$ is perfect.
  We claim first that a.s.\ it has unmatched points of at most one colour.  Denote the ball $S_t:=\{x\in\R^d: |x|<t\}$.  Let $\kappa=\kappa(\gamma)$ be the constant from \cref{quasi}.  Fix $t>0$ and let $T=(2\kappa+1)t$.  Consider the matching $\mat_n$. By \cref{quasi}, if $S_T$ is contained entirely within an $n$-tile then $S_t$ cannot contain points of both colours that do not have partners in $S_T$.  Thus,
  \begin{multline*}
  \P\Bigl(\bigl[\red_n(S_t)-\mat_n(S_t\times S_T)\bigr]
\land
\bigl[\blue_n(S_t)-\mat_n(S_T\times S_t)\bigr]>0\Bigr)\\
  \leq 1-\P(S_T\text{ lies in some $n$-tile}).
  \end{multline*}
  Since the right side tends to $0$ as $n\to\infty$, we deduce that in $\mat$, a.s.\ $S_t$ does not contain points of both colours that do not have partners in $S_T$.  In particular, $S_t$ contains unmatched points of at most one colour.  Since $t$ was arbitrary this proves the claim.
  Now \cref{unm} implies that $\mat$ is perfect.

  Finally, we will show that $\mat$ is $\gamma$-minimal.  First note that by the Skorohod coupling theorem \cite[Theorem~4.30]{kallenberg} we can assume that the convergence \eqref{conv} holds a.s.  Passing to a suitable further subsequence and using the Borel-Cantelli lemma, we can also assume that a.s., for each $t<\infty$, the ball $S_t$ lies in an $n(k)$-tile for all sufficiently large $k$.

  Suppose that $\mat$ is not a.s.\ minimal.  Then with positive probability
  it has some finite set of edges $\{\e{r_i,b_i}\}_{i=1}^N$ whose endpoints
  admit a $2$-colour perfect matching of strictly lower cost.  On this event,
  by continuity of the cost function, there exists (random) $\delta>0$ such
  that, defining the balls $U_i=r_i+S_\delta$ and $V_i=b_i+S_\delta$, for
  any $r'_i\in U_i$ and $b'_i\in V_i$ the matching
  $\{\e{r'_i,b'_i}\}_{i=1}^N$ is also not minimal.
By further reducing
  $\delta$ if necessary, we can assume that no $\overline U_i$ or
$\overline V_i$ contains
  another point of $\red$ or $\blue$ (besides the point $r_i$ or $b_i$ at
  its centre).  Since $\mat_{n(k)}\to \mat$ in the vague topology, and since
  $U_i\times V_i$ is a bounded set with no point of $\mat$ on its boundary,
  for all $k$ sufficiently large, $\mat_{n(k)}$ has points $(r^k_i,b^k_i)\in
  U_i\times V_i$ for each $i=1,\ldots, N$.  But we can find $t<\infty$
  (again, random) such that $r_i,b_i\in S_t$ for all $i$, and then
  $r^k_i,b^k_i\in S_{t+\delta}$.  Then $r^k_i,b^k_i$ belong to the same
  $n(k)$-tile for all $k$ sufficiently large, but this contradicts the
  non-minimality assumption.

  The proof in the $1$-colour case is very similar, with the following
  differences.  We interpret a matching scheme of $\red$ as a point process
  $\mat$ on $(\R^d)^2$, where an edge $\e{x,y}$ is represented by point
  masses at \emph{both} $(x,y)$ and $(y,x)$.  We define $(\mat_n,\red_n)$
  and the limit $(\mat,\red)$ as above.  We can rule out points of the form
  $(x,x)$ in the limit $\mat$, because $\mat_n\{(x,y):x,y\in S_t,\;
  |x-y|<\epsilon\}$ is bounded above by the number of ordered pairs of
  distinct points $x,y$ of a Poisson process that lie in $S_t$ and are at distance at most $\epsilon$, which converges to $0$ in distribution as $\epsilon\to 0$
  for fixed $t>0$.  To show that $\mat$ is a matching scheme of $\R$ we use
  the fact that $\mat_n(U\times V)\leq \red_n(U)\wedge\red_n(V)$ for
  disjoint $U,V\in\A$.  To prove that $\mat$ is perfect we use
  quasistability to show that $S_t$ contains at most one point with no
  partner in $S_T$, and then use \cref{unm}.  Minimality is proved as
  before.
\end{proof}

\begin{remark}
  The quasistability property (\cref{quasi}) is essential for the above
  argument.  In particular, we know by \cref{2colclass}(i) that the
  conclusion of \cref{quasilimit} fails in the case of $2$-colour matching on
  $\R$ with $\gamma>1$.
In that case one may check that
  the limiting matching $\mat$
is empty and thus
has all points unmatched -- indeed, all red
  points are `matched to infinity' in the same direction  (left or
  right) in the limit, and all blue points are matched to infinity in the
  opposite direction.
\end{remark}

\begin{proof}[Proof of \cref{minlimit}]
The argument is a straightforward extension of the proof in
\cite[Section~6]{h} from $\gamma=1$ to $\gamma\in(0,\infty)$, and involves
similar formalism to the proof of \cref{quasilimit} above.  We summarize the
main ideas here, referring the reader to \cite{h} for more detail.
We define the average cost of a perfect matching scheme $\mat$ to be
$$\eta(\mat):=\E\int_{[0,1)^d} |x-\mat(x)|^\gamma \,d\red(x),$$
which equals $\E^* X^\gamma$ for an invariant matching scheme.
Assuming the existence of an invariant matching scheme $\mat$ with
$\eta(\mat)<\infty$, let $I$ be the infimum of $\eta(\mat)$ over all
invariant schemes, and take a sequence of invariant schemes $(\mat_n)$ with
$\eta(\mat_n)<C$  for some $C<\infty$ and $\eta(\mat_n)\to I$.  As in the
previous proof we can take a subsequential limit $\widehat\mat$ in
distribution with respect to the vague topology.  We can use the uniform
bound $\eta(\mat_n)<C$ to conclude that $\widehat\mat$ is perfect and
$\eta(\widehat\mat)=I$.  See \cite[Proof of Corollary~11]{h}. It is important that the cost function
$f_\gamma(x)=x^\gamma$ is
%increasing and
nonnegative and tends to $\infty$ as $x\to\infty$.
Finally, $\widehat\mat$ is $\gamma$-minimal, because otherwise it could be modified locally to obtain another matching with $\eta$ strictly less than~$I$.  See \cite[Proof of Theorem~2(i), case $d\geq 3$]{h}.
\end{proof}

\begin{proof}[Proof of \cref{1colclass}(ii)]
For $\gamma\in(-\infty,1)$ this is a special case of \cref{quasilimit}.  For $\gamma=-\infty$ see \cite{hpps}.
\end{proof}

\begin{proof}[Proof of \cref{higher}]
By \cite{hpps} there exist invariant perfect
matching schemes satisfying the bound
$\P^*(X>t)<c/t$ for $2$ colours in $d=2$, and $\P^*(X>t)<e^{-c'{t^d}}$ for 2
colours in $d\geq 3$ and 1 colour in $d\geq 1$.  Combined with \cref{minlimit}
this covers the claimed cases with $\gamma>0$.  \cref{quasilimit} gives
$\gamma\in(-\infty,1)$.  For $\gamma=-\infty$ see \cite{hpps}.
\end{proof}

\section{Uniqueness and finite differences}
\label{uniq}

In this section we use quasistability and levels to prove uniqueness and
finite differences for $\gamma$-minimal $2$-colour matchings in the
subcritical regime $\gamma<1$ with $d=1$.
(Recall that we do not know whether similar results hold for $1$-colour matchings.)
Here is the key step, a property of the random walk $W$ whose proof we defer until after its applications.

\begin{samepage}
\begin{prop}\label{intervals}
Let $R$ and $B$ be independent Poisson processes of intensity $1$ on $\R$, and define the walk $W$ as in \eqref{walk}.  Fix any $a>1$.  There exists an a.s.\ positive, finite random variable $Y=Y_a$ such that $W>0$ on $[-aY,-Y]\cup[Y,aY]$.  Moreover we can take $Y$ to be supported in the discrete set $\{(3a)^n:n\in\Z^+\}$, and to satisfy the tail bound $\P(Y>y)<y^{-\alpha}$ for some $\alpha=\alpha(a)>0$.
\end{prop}
\end{samepage}

\begin{proof}[Proof of \cref{2colclass}(iv)]
Let $\gamma\in [-\infty,1)$.  There exists a $\gamma$-minimal $2$-colour matching by \cref{quasilimit}.
Let $V=\min((R\cup B)\cap[0,\infty))$ be the first point to the right of the origin.  To establish uniqueness, it suffices to prove that a.s.\ $V$ has the same partner in all $\gamma$-minimal matchings.
Let $\kappa=\kappa(\gamma)$ be as in \cref{quasi}, and let $a=2\kappa+1$.  By \cref{intervals} there exists $Y$ such that $W>0$ on $[-aY,-Y]\cup[Y,aY]$.  This implies that there is some point in $(0,Y]$, so $V$ lies in this interval.  Let $\Lambda$ be the level containing $V$, which is either $\Lambda_0$ or $\Lambda_{-1}$ depending on the colour of $V$.  Then $\Lambda$ has no points in $[-aY,-Y]\cup[Y,aY]$.  Moreover, the set
\begin{equation}\label{hdef}
H:=\Lambda\cap (-Y,Y)
\end{equation}
contains equal numbers of red and blue points.

Recall that a.s.\ every $\gamma$-minimal matching
is perfect (\cref{perfect}).
Let $M$ be any perfect $\gamma$-minimal matching.
We claim that every point in $H$ has its partner in $H$.  If not, since
%the matching is perfect (\cref{perfect}) and
the colours balance, there must exist a red point $r$ and a blue point $b$ in $H$ both with partners outside $H$.  By \cref{levels}, the partners $M(r),M(b)$ are in $\Lambda$ and therefore outside $[-aY,aY]$.  Therefore
$$|r-M(r)|,|b-M(b)|>(a-1)Y=2\kappa Y.$$
But since $|r-b|\leq 2Y$, this contradicts \cref{quasi}.

We have shown that a.s.\ in every $\gamma$-minimal matching, $H$ is matched
to itself.  But $H$ is finite, so by \cref{coincidence} it a.s.\ has only
one $\gamma$-minimal matching, $m$ say.  So every $\gamma$-minimal matching
$M$ has $M(V)=m(V)$.
Hence, there is a.s.\ a unique minimal matching.

Since the minimal matching is unique, it must be a factor, and it is locally infinite by \cref{locinf}.
\end{proof}

We will use the following estimate for the walk, the analogue of a standard fact for simple symmetric random walk in discrete time.

\begin{lemma}\label{hit}
  Let $R$ and $B$ be independent Poisson processes of intensity $1$ on $\R$, and define the walk $W$ as in \eqref{walk}.  The hitting time $T:=\min\{t>0: W(t)=1\}$ satisfies $\P(T>t)\sim c t^{-1/2}$ as $t\to\infty$ for some fixed $c>0$.
\end{lemma}

\begin{proof}
  An application of the reflection principle gives $\P(T\leq t,\, W(t)\leq 0)=\P(T\leq t,\ W(t)\geq 2)$, which implies that $\P(T>t)=\P(W(t)\in\{0,1\})$.  The latter quantity can be analysed by conditioning on the number of jumps of the walk by time $t$ to reduce it to the discrete random walk analogue, and using standard Binomial distribution asymptotics.  Combining the two cases $\{0,1\}$ eliminates parity issues.
\end{proof}

\begin{proof}[Proof of \cref{2coltail}(ii)]
Let us couple the Palm process $(R^*,B^*)$ with $(R,B)$ by adding a red point at the origin: $R^*=R\cup\{0\}$ and $B^*=B$.  Let the walk $W$ be defined in terms of $(R,B)$ via \eqref{walk} as usual, and let $W^*$ be defined similarly in terms of $(R^*,B^*)$, so that $W^*(x)=W(x)+\ind[x\geq 0]$.

First consider the $(1-)$-minimal matching $M=M_{-\infty}$ under the Palm measure.  By \eqref{meshalkin}, the partner of the point at $0$ equals the first return time of $W^*$ to $0$:
\begin{equation}\label{exact}
M^*(0)=\min\{t>0: W^*(t)=0\},
\end{equation}
and moreover, $M^*(0)$ can be determined from the restrictions of $(R^*,B^*)$ to $[0,M^*(0)]$, so the matching is a finitary factor with $L=M^*(0)$.
But the right side of \eqref{exact} equals $\min\{t>0: W(t)=-1\}$, which by symmetry is equal in law to the time $T$ in \cref{hit}.  So the claimed tail bound holds with $\alpha=1/2$.  By symmetry, the same conclusion holds for $M_\infty$.

Now let $\gamma\in[-\infty,1)$.  With $a=2\kappa+1$ as in the proof of \cref{2colclass}(iv) above, let
$$Y:=\min\Bigl\{y:W>0\text{ on }[-ay,-y]\cup[y,ay] \text{ where }y=(3a)^n\text{ for some } n\in\Z^+\Bigr\},$$
and let $Y^*$ be defined similarly in terms of $W^*$.  (The restriction to $y$ of the form $(3a)^n$ avoids complications involving minima versus infima.)  Since $W^*\geq W$ we have $Y^*\leq Y$. By \cref{intervals}, $Y$ satisfies a power law tail bound, therefore so does $Y^*$, and so does $L:=a Y^*$.  We can determine $L$ from the restriction of $(R^*,B^*)$ to $[-L,L]$.  Moreover, by the argument in proof of \cref{2colclass}(iv) (applied to the Palm process, with $V=0$) we can also determine the partner of $0$ in the minimal matching.
\end{proof}

\begin{proof}[Proof of \cref{2colcomp}(ii), case $\gamma,\gamma'<1$]
Here we prove finite differences for the minimal matchings with $\gamma,\gamma'\in(-\infty,1)$.  (The case $\gamma=1-$ will require a different argument, to be given later.)
We apply the same construction as in the proof of \cref{2colclass}(iv)
above, but using the constant $a=2[\kappa(\gamma)\vee \kappa(\gamma')]+1$.
Then the set $H$ defined in \eqref{hdef} is matched to itself in both the
$\gamma$-minimal matching $M$ and $\gamma'$-minimal matching $M'$.  Hence
the component containing $V$ in the graph with edge set $M\cup M'$ is
confined to $H$,
and is thus finite.
\end{proof}

Now we turn to the proof of \cref{intervals}, which we break into lemmas.

%\begin{lemma}\label{stayneg}
%The random walk $F$ satisfies
%$$\limsup_{r\to\infty}\sup_{t\in[0,r]}
%\P\Bigl(F\leq 0 \text{ on }[0,2r-t] \Bigmid F\leq 0 \text{ on }[0,r-t] \Bigr)
%<1.
%$$
%\end{lemma}
%
%\begin{proof}
%Let $p_t=\P(F\leq 0 \text{ on }[0,t])$.  Then $p_t$ is positive and strictly decreasing, and it is standard that $p_t\sim c\, t^{-1/2}$ for some $c>0$ -- see e.g.\ \cite{?}.  The conditional probability in the lemma is
%$$\frac{p_{2r-t}}{p_{r-t}}\leq \frac{p_{3r/2}}{p_{r}}\vee\frac{p_{r}}{p_{r/2}} \to (\tfrac32)^{-1/2}\vee 2^{-1/2}=\sqrt{2/3},$$
%where the maximum arises from splitting into cases $t\leq r/2$ and $t>r/2$.
%\end{proof}

\begin{lemma}\label{stayneg2}
The random walk $W$ satisfies
$$\liminf_{r\to\infty}\inf_{s\in[0,r]}
\P\Bigl(W(s+r)>0  \Bigmid W\leq 0 \text{ \rm on }[0,s] \Bigr)
>0.
$$
\end{lemma}

\begin{proof}
Let
$$T:=\min\{t>0: W(t)>0\}=\min\{t>0: W(t)=1\}.$$
By the strong Markov property at $T$ and symmetry we have
$$\P\Bigl(W(s+r)>0\Bigmid T\Bigr)\geq \tfrac12 \,\ind[T<s+r],$$
and therefore
\begin{align*}&\P\Bigl(W(s+r)>0\Bigmid W\leq 0 \text{ \rm on }[0,s]\Bigr)\\ &=\P\Bigl(W(s+r)>0\Bigmid T>s\Bigr)\geq \tfrac12 \,\P(T<s+r\mid T>s),
\end{align*}

%Let $p_t=\P(T>t)$.  Then $p_t$ is positive and strictly decreasing in $t$, and we claim that $p_t\sim c\, t^{-1/2}$ as $t\to\infty$ for some $c>0$.  (To check this, note that $T$ is the hitting time of $1$, so an application of the reflection principle gives $\P(T\leq t,\, W(t)\leq 0)=\P(T\leq t,\ W(t)\geq 2)$, which implies $\P(T>t)=\P(W(t)\in\{0,1\})$.  The latter quantity can be analysed by conditioning on the number of jumps of the walk by time $t$, to reduce it to the discrete random walk analogue, and using standard Binomial distribution asymptotics).

Write $p_t=\P(T>t)$. Using \cref{hit}, for $s\in[0,r]$ we have
\begin{align*}
\P(T>s+r  \mid T>s )=
\frac{p_{s+r}}{p_{s}}%\\
\leq \frac{p_{r}}{p_{r/2}}\vee \frac{p_{3r/2}}{p_{r}} \xrightarrow{r\to\infty} 2^{-1/2}\vee (\tfrac32)^{-1/2}=\surd\tfrac23,
\end{align*}
where the inequality arises from splitting into the cases $s\leq r/2$ and $s>r/2$.  Therefore the expression in the lemma is at least $\tfrac12 (1-\surd\tfrac23)$.
\end{proof}

%\begin{lemma}\label{frombridge}
%For every $a>1$, the random walk $F$ satisfies
%$$\liminf_{r\to\infty}\inf_{k\geq 0}\P\Bigl(F>0 \text{ \rm on }[3r,3ar]\Bigmid F(2r)=k\Bigr)>0.
%$$
%\end{lemma}
%
%\begin{proof}
%For fixed $a$ and $r$ the probability is clearly increasing in $k$, so it suffices to take $k=0$, in which case it equals
%$$\P\Bigl(F>0 \text{ on } [r,3ar-2r]\Bigr).$$
%By Donsker's theorem \cite{?}, this is at least
%$$\P\Bigl(B>0 \text{ on }[1,3a-2]\Bigr)-o(1)\qquad\text{as }r\to\infty,$$
%where $B$ is standard Brownian motion.  The last probability is $\pi^{-1}\arcsin [(3a-2)^{-1/2}]$, by symmetry and the arcsine law for the last zero \cite{?}.
%\end{proof}

\begin{lemma}\label{frombridge2}
For any fixed $1<u<v$, the random walk $W$ satisfies
$$\liminf_{r\to\infty}\inf_{k\geq 0}\P\Bigl(W>0 \text{ \rm on }[ur,vr]\Bigmid W(r)=k\Bigr)>0.
$$
\end{lemma}

\begin{proof}\sloppypar
For fixed $u,v$ and $r$ the probability is clearly increasing in $k$, so it
suffices to take $k=0$, in which case by the Markov property it equals
\begin{align}\label{PW>}
\P\Bigl(W>0 \text{ on } \bigl[(u-1)r,(v-1)r\bigr]\Bigr).
\end{align}

As $r\to\infty$,
the rescaled process $(r^{-1/2}\,W(rt))_{t\ge0}$ converges in distribution
in the Skorohod topology on $D[0,\infty)$
to $\sqrt2 B(t)$, where $B$ is standard Brownian motion;
this is an easy consequence of Donsker's theorem \cite[Theorem~14.9]{kallenberg}
and a random time change;
alternatively it follows directly by general limit theorems for L\'evy processes
\cite[Theorems 15.14 and 15.17]{kallenberg}.
Consequently, the probability in \eqref{PW>} converges to
\begin{align*}\label{brown}
\P\Bigl(B>0 \text{ on }[u-1,v-1]\Bigr)
%-o(1)\qquad\text{as }r\to\infty
.\end{align*}
By symmetry and the arcsine law for the last zero \cite[Theorem~13.16]{kallenberg}, the last probability equals $\pi^{-1}\arcsin \sqrt{(u-1)/(v-1)}$, which is positive.
\end{proof}

The essential infimum, $\essinf X$, of a random variable $X$ is the largest constant $x$ such that $X\geq x$ a.s.

\begin{cor}\label{step}
Define the random set $S(r)=\{x\in(0,r]:W(x)>0\}$.  For each $a>1$,
$$\liminf_{r\to\infty}\essinf\P\Bigl(W>0 \text{ \rm on } [3r,3ar]\Bigmid S(r)\Bigr)>0.$$
\end{cor}

\begin{proof}[Proof of \cref{step}]
We use the Markov property at the intermediate time $2r$.
By \cref{frombridge2} we have
$$\liminf_{r\to\infty}\inf_{k> 0}\P\Bigl(W>0 \text{ \rm on }[3r,3ar]\Bigmid W(2r)=k\Bigr)>0.
$$
Therefore, defining $D_r=\{W(2r)> 0\}$, it is enough to prove that
\begin{equation}\label{dbound}
\liminf_{r\to\infty}\essinf\P\bigl(D_r\bigmid S(r)\bigr)>0.
\end{equation}

Let $L= \sup (S(r)\cup\{0\})$, and note that
$$\P(D_r\mid S(r))=\P\bigl(D_r\bigmid L,\,W(L)\bigr).$$
We consider two cases.  If $L=r$ then $W(r)>0$.  But for any integer $k>0$ we have
$$\P\bigl(D_r\bigmid L=r,\,W(L)=k\bigr)=\P\bigl(D_r\bigmid W(r)=k\bigr)\geq \P\bigl(D_r\bigmid W(r)=1\bigr)\geq \tfrac12$$
by symmetry.
On the other hand, if $L\in[0,r)$ then $W(L)=0$ and $W\leq 0$ on $[L,r]$.  Moreover, for $t<r$,
\begin{align*}
\P\bigl(D_r\bigmid L=t\bigr)
&=\P\bigl( W(2r)>0\bigmid W(t)=0,\,W\leq 0 \text{ on }[t,r]\bigr)\\
&= \P\Bigl(W(2r-t)>0  \Bigmid W\leq 0 \text{ on }[0,r-t]\Bigr)
%&\ge1-\P\Bigl(F\leq 0 \text{ on }[0,2r-t] \Bigmid F\leq 0 \text{ on %}[0,r-t]\Bigr),
\end{align*}
which by \cref{stayneg2} (with $s=r-t$) is bounded away from $0$ as $r\to\infty$ uniformly in $t$.
Combining the two cases we obtain \eqref{dbound}.
\end{proof}

\begin{proof}[Proof of \cref{intervals}]
Fix $a>1$, and for integer $n\geq 1$ define events
\begin{align*}
  A^+_n&:=\bigl\{W>0 \text{ on } [(3a)^{n},a(3a)^{n}]\bigr\};\\
  A^-_n&:=\bigl\{W>0 \text{ on } [-a(3a)^{n},-(3a)^{n}]\bigr\},
\end{align*}
and write
$
  A_n:=A^+_n\cap A^-_n
%  \qquad B_n:=\bigcap_{i=1}^n \overline{A_n}.
$.
%We will show that $\P(\cup_{n=1}^\infty A_n)=1$, which implies the desired %result.

Define the random sets
\begin{align*}
S^+_n&:=\bigl\{x\in (0,a(3a)^n]: W(x)>0\bigr\};\\
S^-_n&:=\bigl\{x\in [-a(3a)^n,0): W(x)>0\bigr\},
\end{align*}
and let $S_n=S^+_n\cup S^-_n$.  By \cref{step} there exist $\delta=\delta(a)>0$ and $N=N(a)<\infty$ such that for all $n\geq N$,
$$\P(A_{n+1}^+\mid S^+_n)>\delta\qquad\text{a.s.}$$
By symmetry and independence of the left and right half lines it follows that for $n\geq N$,
$$\P(A_{n+1}\mid S_n)>\delta^2\qquad\text{a.s.}$$
In particular,
$$\P\Bigl(A_{n+1}\Bigmid \textstyle \bigcap_{i=1}^n \overline{A_n}\Bigr)>\delta^2$$
for $n\geq N$, and since this conditional probability is positive for all $n$, it is bounded below for all $n\geq 0$.  We deduce that $\P(\cup_{n=1}^\infty A_n)=1$, and moreover that $\P(\cap_{i=1}^n \overline{A_n})$ decays exponentially as $n\to\infty$.  Since on $A_n$ we can set $Y=(3a)^n$ we obtain the claimed result.
\end{proof}

\section{Edge Orientations in the subcritical case}
\label{orient}

In this section we prove a structural property of the $2$-colour $\gamma$-minimal matching
for $\gamma<1$.  In addition to being interesting in its own right, this
will enable us to prove finite differences between $\gamma<1$ and
$\gamma=1-$.  Recall that, for $\gamma<1$, each point is matched within its
own level (\cref{levels}).

\begin{thm}[Locally infinite levels]\label{nest}
  Let $R$ and $B$ be independent Poisson processes of intensity $1$ on $\R$.  Let $\gamma\in[-\infty,1)$ and consider the $\gamma$-minimal $2$-colour matching.  Almost surely, for any $x\in\R$ and $k\in\Z$, the level $\Lambda_k$ contains infinitely many edges that cross $x$.  These edges can be ordered as $e_1,e_2,\ldots$, where $e_{i+1}$ straddles $e_i$ for each $i\geq 1$; then their orientations alternate.
\end{thm}

\begin{proof}
  The argument will be similar to the proof of \cref{2colclass}(iii) concerning $\gamma=1-$ (but more intricate).  Let $M$ be the minimal matching, which is unique by \cref{2colclass}(iv) and perfect by \cref{perfect}.  No two edges are entwined, so each point is matched within its level by \cref{levels}.

  Fix $k\in\Z$ and let $m_k=M|_{\Lambda_k}$ be the restriction of the matching to level $k$.
  Since bounded intervals contain finitely many points, either $m_k$ is locally finite or every $x\in\R$ is crossed by infinitely many edges of $m_k$.
  Consider any location $x\in\R\setminus(R\cup B)$ and let $e_1,e_2,\ldots,e_n(,\ldots)$ be the edges of $m_k$ that cross $x$, ordered so that $e_{i+1}$ straddles $e_i$ for each $i$.  We claim that their orientations alternate.  Indeed, suppose that $e'=\e{r',b'}$ straddles $e=\e{r,b}$, and that they belong to the same level and have the same orientation, say right.  Since the points of the level alternate in colour, it has more blue than red points in the interval $(r',r)$.  Since edges do not entwine, at least one of these blue points must be matched to a red point between $b$ and $b'$, giving an edge of $m_k$ of the opposite orientation straddling $e$ and straddled by $e'$.  This proves the claim.

  Call an edge of $m_k$ \df{outer} if it not straddled by any other edge of $m_k$.  If $m_k$ is locally finite then it has infinitely many outer edges, and, since the colours alternate, all its outer edges must have the same orientation; $m_k$ partitions $\R$ into an alternating sequence of outer edges and \df{gaps}, i.e.\ intervals crossed by no edges.  (The outer edges may straddle other edges of $m_k$.)  We say that $m_k$ has \df{left type} if all its outer edges are oriented left, or \df{right type} if they are oriented right, or \df{$\infty$ type} if it is locally infinite.

  Now we consider how different levels are related.  Suppose that $m_k$ has
  right type, and consider a gap, which is an interval $I=(b,r)$ where
  $b,r\in\Lambda_k$ with $b$ blue and $r$ red, and no edges of $m_k$
  crossing $I$.  The walk $W$ satisfies $W\leq k$ on $I$.  In particular $I$
  contains no points of $\Lambda_{k+1}$.  Moreover, since the same applies
  to each gap and there
are no entwined edges, $I$ is crossed by no edges of $m_{k+1}$.
  On the other hand, a.s., some outer edge of $m_k$ must cross some point of $\Lambda_{k+1}$, otherwise we would have $W\leq k+1$ on $\R$.  Suppose that $\e{r,m_k(r)}$ is such an outer edge.
  Let $r'$ be the first point of $\Lambda_{k+1}$ to the right of $r$, which must be red because $W(r+)=k+1$.  The partner $b'=m_{k+1}(r')$ must be to the right of $r'$, since edges do not entwine, and by the previous remarks, $\e{r',b'}$ is an outer edge.  Therefore, $m_{k+1}$ also has right type.  A similar argument shows that if $m_k$ has left type then so does $m_{k-1}$.

  We conclude that there exist (random) $K_-,K_+\in\Z\cup\{-\infty,\infty\}$ with $K_-\leq K_+$ such that for each $j\in\Z$,
  %$$m_j\text{ has type:}\begin{cases}
%    \text{right},& K_+-\tfrac12<j\\
%    \infty,& K_--\tfrac12<j<K_+-\tfrac12\\
%    \text{left},& j<K_--\tfrac12
%  \end{cases}.$$
  $$m_j\text{ has type:}\left\{\begin{array}{l@{\quad}r@{\ }l}
    \text{right},& K_+-\tfrac12<&j,\\[2pt]
    \infty,& K_--\tfrac12<&j<K_+-\tfrac12,\\[2pt]
    \text{left},& &j<K_--\tfrac12.
  \end{array}\right.$$
  So far we have used that $M$ is perfect and has no entwined edges, and that each level is unbounded in both directions.  Next we will use invariance properties to show that $K_-=-\infty$ and $K_+=\infty$.

  First we claim that all levels have the same type; the idea is that it is impossible to specify a level in an invariant way.
  Recall the $(1-)$-minimal matchings $(M_k:k\in\Z\cup\{-\infty,\infty\})$ defined in \eqref{defmk}.  Construct the new matching $M_{K_-}$, where $K_-$ is the random variable above.  This is a perfect matching of $R$ and $B$, and it can be constructed as a translation-equivariant function of $M$: within each level $j$ we replace the matching $m_j$ with one of the two matchings from the earlier proof: either $m_j^-$ (if $m_j$ has left type) or $m_j^+$ (if $m_j$ has right or $\infty$ type).  Therefore, $M_{K_-}$ is an invariant matching scheme, so by \cref{locinf}, it is locally infinite a.s.  But as shown in the proof of \cref{2colclass}(iii), $M_k$ is locally finite for finite $k$.   Therefore, $K_-\in\{-\infty,\infty\}$ a.s.  Similarly, considering the matching $M_{K_+}$ shows that $K_+\in\{-\infty,\infty\}$ a.s.  Thus, a.s., all levels of $M$ have the same type.

  Since the $\gamma$-minimal matching $M$ is unique, it is a factor, and thus ergodic.
  The event that all levels have left type is translation invariant, and similarly for right and type $\infty$ types.  Therefore for one of the three types, a.s.\ all levels have that type.  Finally, we can rule out left type and right type, because $M$ is invariant in law under reflections of $\R$.  Thus all levels are locally infinite, completing the proof.
\end{proof}

We can now prove finite differences in the remaining case.

\begin{proof}[Proof of \cref{2colcomp}(ii), case $\gamma'=1-$]\sloppypar
Let $\gamma\in[-\infty,1)$ and let $M$ be the $\gamma$-minimal matching.
Let $M'$ be one of the $(1-)$-minimal matchings $(M_k:
k\in\Z\cup\{-\infty,\infty\})$.  Since both $M$ and $M'$ match within
levels, it is enough to prove finite differences within a level, say
$\Lambda_j$.  Recall that the restriction of $M'$ to $\Lambda_j$ is one of
the two alternating matchings, say without loss of generality $m_j^+$.  Fix
a point $x\in\Lambda_j$.  By \cref{nest} there exists a right-oriented edge
$\e{r,b}\in M|_{\Lambda_j}$ with $r<x<b$.  Since all edges of $m_j^+$ are
oriented right, $r$ and $b$ are matched in the interval $[r,b]$ in $M'$.
Since there are no entwined edges, every point in $\Lambda_j\cap[r,b]$ is
matched within this set in both $M$ and $M'$.
\end{proof}

\begin{remark}
  Notwithstanding the above proof, the finite differences property does \emph{not} hold between any two of the distinct $(1-)$-minimal $2$-colour matchings $M_k$, as is easily checked by considering a level on which they differ.
\end{remark}

\section{Tail Bounds}
\label{tail}

\begin{proof}[Proof of \cref{2coltail}(i)]
  The condition $\E^* X^{1/2}=\infty$ holds for \emph{any} invariant $2$-colour matching scheme of $R$ and $B$, by \cite[Theorem~2]{hpps}.  We therefore turn to the claimed upper tail bound.

  For $\gamma=1-$ and $M=M_{-\infty}$ the argument was already given in the proof of \cref{2coltail}(ii): by \eqref{meshalkin}, $X=M^*(0)$ is the first return time of the walk $W^*$ to $0$ \eqref{exact}, so the claimed bound holds by \cref{hit}.  The case $M=M_\infty$ is similar.

 Let $\gamma\in [-\infty,1)$ and let $M$ be the unique $\gamma$-minimal $2$-colour matching in $d=1$.
  Let $\kappa=\kappa(\gamma)$ be the constant from \cref{quasi}.  Let $t>0$ and call a red or blue point $x$ \df{bad} if $|x-M(x)|>\kappa t$, and \df{good} otherwise.  By \cref{quasi}, the interval $[0,t]$ cannot contain bad points of both colours.

  Suppose that $[0,t]$ contains bad red points, and let $U$ and $V$ be the
  first and last bad red point in the interval.  Then no good point in
  $[U,V]$ is matched outside $[U,V]$, since that would entail entwined
 edges, in contradiction to \cref{ineq} (see also \cref{possibilities}).
Hence $[U,V]$ contains good red and good blue points in equal numbers, and so the number of bad red points in $[0,t]$ equals $W(V+)-W(U-)$, where $W$ is the random walk of \eqref{walk}.  We deduce that
  \begin{align*}
   t\,\P^*(X>\kappa t) &=\E\#\bigl\{x\in R\cap[0,t]: x \text{ is bad}\bigr\} \\
   &\leq \E\sup_{u,v:\,0<u\leq v<t} \bigl(W(v)-W(u)\bigr)^+ \\
   &\leq 2\,\E\sup_{x\in[0,t]} |W(x)|
   \leq C\,t^{1/2},
  \end{align*}
  for some fixed $C>0$.  (To check the last inequality, Doob's martingale inequality \cite[Proposition~7.16]{kallenberg} gives $\|\sup_{x\in[0,t]}|W(x)|\|_2\leq 2 \|W(t)\|_2=2\sqrt{2t}$, whereupon Lyapunov's norm inequality completes the argument).
  The bound $\P^*(X>x)<c\,x^{-1/2}$ now follows.
\end{proof}

For $1$-colour matching we have an upper bound in all dimensions.

\begin{thm}\label{dmom}
  Let $d\geq 1$ and $\gamma\in [-\infty,1)$ and let $R$ be a Poisson process of intensity $1$ on $\R^d$.  There exists $c=c(d,\gamma)>0$ such that for any invariant $\gamma$-minimal $1$-colour matching scheme we have
  $$\P^*(X>x)<c\,x^{-d},\qquad x>0.$$
\end{thm}

\begin{proof}
  Let $\kappa=\kappa(\gamma)$ be the constant from \cref{quasi}.  Let $t>0$ and call a red or blue point $x$ \df{bad} if $|x-M(x)|>\kappa t$, and \df{good} otherwise. By \cref{quasi}, the ball $S_{t/2}=\{x\in\R^d :|x|<t/2\}$ contains at most one bad point.  Therefore, writing $\omega=\omega(d)$ for the volume of the unit ball,
  \begin{equation*}
  (t/2)^d \,\omega\, \P^*(X>\kappa t) =\E\#\bigl\{x\in R\cap S_{t/2}: x \text{ is bad}\bigr\}
  \leq 1. \qedhere
  \end{equation*}
\end{proof}

\begin{proof}[Proof of \cref{1coltail}]
Part (i) ($\gamma\geq 1-$) is a trivial consequence of the proof of
\cref{1colclass}(i).
The invariant matching scheme $M$ is an equal mixture of the two alternating matchings $M_+$ and $M_-$.  Therefore the same is true for the Palm version, and thus $X$ is a standard exponential variable.
Part (ii) ($\gamma<1$)
%$\P^*(X>x)<c\,x^{-1}$
 is the $d=1$ case of \cref{dmom}.
\end{proof}

\section*{Acknowledgements}

We thank Maria Deijfen for many valuable conversations.  Alexander Holroyd thanks the University of Uppsala for a most enjoyable visit during which much of this work was carried out.

%Monotonicity for two-colour quasi-stable cases???

%\section{Summary}
%
%\begin{center}
%\begin{tabular}{|l|c|c|c|c|}
%\hline
%\multicolumn{5}{|c|}{\textbf{One-colour $p$-minimal matchings on $\R$}} \\
%  \hline
%  $p$&exists&invariant&factor&a.s.\ unique\\
%  \hline
%  % after \\: \hline or \cline{col1-col2} \cline{col3-col4} ...
%  $-\infty$ (stable) & y & y & y & y \\
%  $(-\infty,0)$ &  &  &  &  \\
%  $0$ (logarithmic) &  &  &  &  \\
%  $(0,1)$ &y & y &  &  \\
%  $1$ & y & y & n & n \\
%  $(1,\infty)$ & y & y & n & n \\
%  $\infty$ (altruistic) & y & y & n & n \\
%  \hline
%\end{tabular}
%
%\bigskip
%
%
%\begin{tabular}{|l|c|c|c|c|c|}
%\hline
%\multicolumn{6}{|c|}{\textbf{Two-colour $\gamma$-minimal matchings on $\R$}} \\
%  \hline
%  $\gamma$&exists&perfect&invariant&factor&a.s.\ unique\\
%  \hline
%  $[-\infty,1)$ & \y & \y & \y & \y & \y\\
%  $1^-,1$       & \y & \y & \y & \y & \n\\
%  $[1^+,\infty]$& \y & \y & \n & \n & \n \\
%  \hline
%\end{tabular}
%\end{center}
%
%
%
%
%\section{Questions}
%
%
%
%\subsection*{Other}
%
%Assume uniqueness compare two different $p<1$.  Are the components of the union of the two matchings finite?
%
%Are the $p$ powers the only scale invariant choices?
%
%Say something about limits of finite matchings.  Almost sure? In dist? Subsequential? Boundary conditions?
%
%Game interpretation (Johan)?
%
%Heirarchical models and PWITS
%
%Multicolour, stubs etc?

\bibliographystyle{abbrv}
\bibliography{mm}

\end{document}